\documentclass[a4paper,12pt%
]{amsart}
\pdfoutput=1
\usepackage{amsmath, amssymb, amsthm, amscd}
\usepackage{graphicx}
\usepackage{enumerate}
\usepackage[all,cmtip]{xy}

\newcounter{theorem}
\newtheorem{thm}[theorem]{Theorem}
\newtheorem{lemma}[theorem]{Lemma}
\newtheorem{prop}[theorem]{Proposition}
\newtheorem{cor}[theorem]{Corollary}
\newtheorem{defn}[theorem]{Definition}

\theoremstyle{remark}
\newtheorem*{remark*}{Remark}
\newtheorem{remark}[theorem]{Remark}
\newtheorem{example}[theorem]{Example}

\numberwithin{equation}{section}
\numberwithin{theorem}{section}

\DeclareMathOperator\spa{span}

\DeclareMathOperator\adj{adj}
\DeclareMathOperator\Tr{Tr}

\newcommand{\defemph}{\emph}

\newcommand{\R}{\mathbb{R}}
\newcommand{\Q}{\mathbb{Q}}
\newcommand{\Z}{\mathbb{Z}}

\newcommand{\N}{\mathbb{N}}

\newcommand{\twotwo}[4]{\left( \begin{array}{cc} #1 & #2 \\ #3 & #4 \end{array}\right)}

\title[$p$-adic functionals]{$p$-adic functionals on torsion-free abelian groups}

\author{Gregory R. Maloney}
\address{Newcastle University}
\date{\today}

\begin{document}

\begin{abstract}
Let $p$ be a prime.  
A $p$-adic functional on a torsion-free abelian group $G$ is a group homomorphism from $G$ to the $p$-adic integers.  
The group of all such $p$-adic functionals is viewed as a $p$-adic dual group of $G$, and is studied from the point of view of functional analysis.  
An analogue of the Hahn--Banach Theorem is proved; this result shows that there are sufficiently many $p$-adic functionals to be interesting.  

There is a natural homomorphism from $G$ to its $p$-adic double dual, and one main result that is proved is that the image of $G$ in this double dual is dense in an appropriate topology.  
This is used to prove the second main result, which says that the $p$-adic double dual of $G$ is the same as the pro-$p$ completion of $G$.  

The theory of $p$-adic functionals can then be used to produce a matrix description of $G$ if $G$ has finite rank.  
This matrix description is related to other matrix descriptions due to Kurosch and Malcev.  
\end{abstract}

\maketitle

\begin{quotation}
\textit{Es w\"{a}re sehr n\"{u}tzlich, noch einige tieferliegende Verbindungen zwischen $p$-adischen Zahlen und primitiven Gruppen zu finden---sie existieren wahrscheinlich---um diese Theorie mehr begrifflich zu machen und insbesondere schwierige Berechnungen} [\ldots] \textit{zu eliminieren.  
}
\end{quotation}
{\hfill ---Alexander Kurosch \cite{K:matrices}}

\section{Introduction}\label{SEC:intro}

Let $p$ be a prime and let $G$ be a torsion-free abelian group.  
The subject of this work is the group $Hom (G,\Z_p)$ of $p$-adic functionals on $G$.  
This group is viewed as a $p$-adic dual of $G$, and is studied using the techniques of functional analysis.  
This is a novel approach that adds clarity to a difficult area of study.  
To emphasize this approach, the notation $G^{*p}$ is used instead of $Hom (G,\Z_p)$.  
It is shown that an analogue of the Hahn--Banach Theorem (Proposition \ref{PROP:all-states-extend}) holds in this setting; broadly speaking, this means that $G^{*p}$ contains enough functionals to be interesting.  

There is a natural homomorphism of $G$ into its $p$-adic double dual, $G^{*p*p}$, and one of the main results of this work is that the image of $G$ is dense in the relative topology that the double dual inherits as a subspace of $\Z_p^{G^{*p}}$ (Theorem \ref{THM:density}).  
As a submodule of the free $\Z_p$-module $\Z_p^{G^{*p}}$, $G^{*p*p}$ is also a free $\Z_p$-module.  
Thus, in particular, it is a pro-$p$ group, and the pro-$p$ topology agrees with the relative topology.  

There is another pro-$p$ group, called the pro-$p$ completion of $G$, that shares this same property of having a homomorphism from $G$ into it, the image of which is dense.  
Moreover, the pro-$p$ completion satisfies a universal property, which implies that there exists a continuous homomorphism from the pro-$p$ completion to the $p$-adic double dual.  
The main result of Section \ref{SEC:completions} is Theorem \ref{THM:pro-p}, which states that this homomorphism has a continuous inverse, so the pro-$p$ completion and the $p$-adic double dual of $G$ are continuously isomorphic (in a way that preserves the image of $G$ in each of them).  

The $\Z_p$-modules $G^{*p}$, as $p$ ranges over the set of all primes, give enough information to reconstruct the isomorphism class of the torsion-free abelian group $G$.  
Section \ref{SEC:classification} describes how this is done in the finite-rank case; in particular, if $G$ has finite rank, then Theorem \ref{THM:classification-3} provides an isomorphism invariant for $G$.  
(For a torsion-free abelian group $G$, let us use the term \defemph{rank} to refer to the torsion-free rank of $G$, that is, the size of any maximal independent subset.)  
This invariant is a matrix description of $G$ that is related to others due to Kurosch and Malcev from 1937 and 1938 respectively (see \cite{K:matrices} and \cite{M:matrices}, or \cite{K:summary} for a more modern description of the latter, with applications to developments since then).  

\subsection{History}\label{SUBSEC:history}

This work contributes to the classification and description of torsion-free abelian groups, which is an area with a long history---an excellent overview of the theory is given in \cite{A:book}.  
In 1937, Baer \cite{B:rank-one} classified the torsion-free abelian groups of rank one using an invariant called the \defemph{type} of the group, which is an equivalence class of supernatural numbers.  

At the same time, Kurosch \cite{K:matrices} and Malcev \cite{M:matrices} developed matrix descriptions of a finite-rank torsion-free abelian group.  
Definitions \ref{DEF:group-determined} and \ref{DEF:factored} and Theorem \ref{THM:classification-3} from Section \ref{SEC:classification} provide a matrix description that is related to these.  

The results here are most closely related to work of Warfield from 1968 \cite{W:homomorphisms}, in which, for a rank-$n$ group $M$ and a rank-$1$ group $G$, Warfield considers the dual $Hom (M,G)$ and the double dual $Hom($ $Hom ($ $M,G),G)$, and gives necessary and sufficient conditions for the natural homomorphism from $M$ to the double dual to be an isomorphism.  

The functional analysis approach that is used here is modelled after the theory of functional analysis on partially ordered torsion-free abelian groups, an excellent reference for which is the book \cite{G:book}.  
In particular, the $p$-adic Hahn--Banach Theorem that is proved here is modelled after Lemma 4.1 and Proposition 4.2 of \cite{G:book} (originally from \cite{GH:states}).  

The matrix description of finite-rank torsion-free abelian groups that appears in Section \ref{SEC:classification} is one of the main results of this work (although by no means is it the only one).  
Given that it is closely related to the matrix descriptions of Kurosch and Malcev, it would be natural to ask: What is new in this work that is not already present in the works of those authors?  
In response to this question, I would point to the epigraph that precedes this introduction.  
This epigraph is a quotation from Kurosch, in which he acknowledges that his matrices arise from a difficult sequence of calculations, and calls for a better theoretical understanding of the relationship between torsion-free abelian groups and $p$-adic integers.  
The theory of $p$-adic functionals that is introduced here subsumes most of the difficult calculations, making them invisible, and so represents a clear answer to Kurosch's entreaty.  
There is the added bonus that the matrices in the present work are simpler than Kurosch's or Malcev's in that they have fewer rows.  

Furthermore Theorem \ref{THM:pro-p}, about the pro-$p$ completion of a torsion-free abelian group, is another consequence of this theory, and to my knowledge this result is entirely new.  

\subsection{Notation}\label{SUBSEC:notation}

Let $p$ be a prime.  
$\Z_p$ denotes the ring of $p$-adic integers and $\Q_p$ denotes its field of fractions.  
The $p$-adic absolute value of an element $\alpha$ of $\Q_p$ or $\Z_p$ is denoted by $|\alpha |_p$.  
If $n\in\N$ and $R$ is a ring, then $R^n$ denotes the module consisting of $n\times 1$ columns with entries in $R$, while $(R^n)^*$ denotes its dual consisting of $1\times n$ rows.  
If $m,n\in\N$, then $M_{m\times n}(R)$ denotes the module of $m\times n$ matrices over $R$.  
For $r\in R^n$, $r^t$ denotes its transpose in $(R^n)^*$.  
For $A\in M_{m\times n}(R)$, $A^t$ denotes its transpose in $M_{n\times m}(R)$.  
For $v = (v_1,\ldots, v_n)^t \in \Z_p^n$, the norm $\| v\|_p $ is defined by $\| v\|_p := \max_{1\leq i\leq n} \{ \| v_i\|_p\}$.  
Let $G$ be a torsion-free abelian group and $S$ a subset of $G$.  
Then $\langle S\rangle_*$ denotes the pure subgroup of $G$ generated by $S$; that is, $\langle S\rangle_* := \{ g\in G : ng \in \langle S\rangle$ for some $n\in\N\}$.  

\subsection{Outline}\label{SUBSEC:outline}

Section \ref{SEC:examples} contains examples of torsion-free abelian groups with calculations of their $p$-adic functionals.  

Section \ref{SEC:states} is the biggest section, and introduces the notation and terminology of $p$-adic functionals.  
It includes a discussion of the $p$-adic topology on a torsion-free abelian group $G$ (Section \ref{SUBSEC:pseudometric}); the statement and proof of an analogue of the Hahn--Banach Theorem (Proposition \ref{PROP:all-states-extend} in Section \ref{SUBSEC:functionals}); results about the dual space of all $p$-adic functionals on $G$ (Section \ref{SUBSEC:state-spaces}); and results about double dual spaces, including the fact that the natural homomorphism from $G$ to its double dual has dense image (Theorem \ref{THM:density} in Section \ref{SUBSEC:double-duals}).  

Section \ref{SEC:completions} contains a proof that the $p$-adic double dual of a torsion-free abelian group $G$ can be naturally identified with the pro-$p$ completion of $G$ (Theorem \ref{THM:pro-p}).  

Section \ref{SEC:classification} uses the theory of $p$-adic functionals to classify finite-rank torsion-free abelian groups up to isomorphism; this is the content of Theorem \ref{THM:classification-3}.  
The classifying invariant is a sequence of matrices that is related to Malcev's matrices, and is subject to the same equivalence relation as Malcev's matrices \cite{M:matrices}, although these matrices have fewer rows.  

\section{Examples}\label{SEC:examples}

This section contains two examples to illustrate the theory.  
The groups in both of these examples are inductive limits of stationary inductive sequences of finitely-generated free abelian groups---see \cite{D:stationary} for previous results about these groups, including a classification up to quasi-isomorphism.  
There has already been work done on the real-valued functionals on \emph{ordered} abelian groups that are inductive limits of stationary inductive sequences of simplicial groups (see \cite{H:irrational} and \cite{M:stationary}).  
The $p$-adic theory is similar to the theory in the ordered case: in particular, $p$-adic functionals come from rows that are combinations of certain left eigenvectors of a matrix representation of the homomorphism from the stationary inductive sequence \cite[Theorem 4.9]{M:finite-rank}.  

\begin{example}\label{EX:favourite}
Consider the matrix $A = \twotwo{1}{1}{1}{4}$.  
The characteristic polynomial of $A$ is $\chi_A(x) = x^2-5x+3$, which has roots $x = \frac{5\pm\sqrt{13}}{2}$.  
These roots, being algebraic integers, have $3$-adic absolute value at most $1$, and, as their product is $\det A = 3$ and their sum is $\Tr A = 5$, exactly one of them must have $3$-adic absolute value less than $1$.  
Let us use the symbol $\lambda$ to denote this eigenvalue.  
Then eigenvectors of $A$ associated to $\lambda$ and its algebraic conjugate, $5-\lambda$, can be taken to be $v_1 := {{\lambda - 4}\choose{1}}$ and $v_2 := {{1-\lambda}\choose{1}}$.  

Consider the torsion-free abelian group $G = \bigcup_{n=0}^\infty A^{-n}(\Z^2)$.  
$\Z$ is dense in $\Z_3$, so, as $\Z^2\subseteq G$, $G$ contains elements that are arbitrarily close to $v_1$ and $v_2$ in the norm $\| \cdot \|_3$.  
In particular, the sequence 
\[
{{-1}\choose{1}}, {{-7}\choose{1}}, {{-7}\choose{1}}, {{-34}\choose{1}}, {{-115}\choose{1}}, {{128}\choose{1}}, \ldots 
\]
has the property that its $n$th element is divisible by $3^n$ in $G$, and this sequence converges to the $\lambda$-eigenvector ${{\lambda - 4}\choose{1}}$ of $A$.  

There is a metric on $G$ called the $3$-adic metric (Definition \ref{DEF:pseudometric}) that says that an element is small if it is divisible in $G$ by high powers of $3$.  
Write $g\in G$ as $g = \alpha_1v_1+\alpha_2v_2$; then it can be proved that $g$ is divisible by high powers of $3$ if and only if $|\alpha_2|_3$ is small.  
(Here is the idea of the proof:  $g\in 3^nG$ if and only if $G^kg \in 3^n\Z^2$ for some $k\geq 0$.  
But $G^kg = \lambda^k\alpha_1v_1 + (5-\lambda)^k\alpha_2v_2$, and $|\lambda^k|_3 = 3^{-k}$, which is already small in $p$-adic absolute value, so $G^kg$ will be small in the $3$-adic sense if and only if $|(5-\lambda)^k\alpha_2|_3 = |\alpha_2|_3$ is small.)  

Any group homomorphism $f : G\to \Z_3$ can be represented as left multiplication by a row in $(\Z_p^2)^*$, and any such row has the form $\beta_1v_1^t + \beta_2v_2^t$, where $\beta_1,\beta_2\in \Z_3$.  
As $A$ is symmetric, $v_1^t$ and $v_2^t$ are left eigenvectors associated to $\lambda$ and $5-\lambda$ respectively.  
The fact that $f: G\to \Z_3$ is a group homomorphism implies that it must be continuous with respect to the $3$-adic metric on $G$, which happens if and only if $\beta_1 = 0$ (this is a consequence of the fact that $v_1^tv_2 = 0$, so the $v_1^t$ component of $f$ kills the component of $g$ that determines how $3$-divisible $g$ is).  

Thus, up to multiplication by a scalar, there is a unique group homomorphism $f : G\to \Z_3$, which is given by left multiplication by $v_2^t$.  
The image under $f$ of $G$ is 
\begin{align*}
(1-\lambda, 1) \bigcup_{n=0}^\infty A^{-n}(\Z^2) & = (5-\lambda)^{-n}(1-\lambda,1)\Z^2 \\
& = \{ (5-\lambda)^{-n} [ (1-\lambda)a+b] : n\geq 0, a,b\in\Z\} \\
& = \Big\{ \Big( \frac{\lambda}{3}\Big)^n (c\lambda+d) : n\geq 0, c,d\in\Z\Big\}.
\end{align*}

$f$ is an isometry, so $G$ is in fact isomorphic to this image, which is a subgroup of $\Z_3$ because $\lambda/3\in\Z_3$.  

\end{example}

The group in the next example admits two independent $3$-adic functionals.  
\begin{example}\label{EX:rank-two-dual}
Consider the matrix $B = \left( \begin{array}{rrr} 0 & 0 & -3 \\ 1 & 0 & -1 \\ 0 & 1 & -1 \end{array}\right)$.  
The characteristic polynomial of $B$ is $\chi_B(x) = x^3+x^2+x+3$; let us denote the roots of this by $\lambda, \mu,$ and $\nu$.  
The product of these roots is $\det B = -3$, so at least one of them---say $\lambda$---has $3$-adic absolute value less than $1$.  
But $\lambda\mu + \mu\nu + \nu\lambda = 1$, so $|\mu|_3= |\nu|_3 = 1$.  
It is a consequence of the Newton Polygon Theorem \cite[Theorem 6.3.1]{C:local-fields} that $\lambda\in \Q_3$, and hence $\lambda\in \Z_3$, although the Theorem does not say anything $\mu$ and $\nu$.  
$w_\lambda := (1,\lambda, \lambda^2)$, $w_\mu := (1,\mu,\mu^2)$, and $w_\nu := (1,\nu,\nu^2)$ are left eigenvectors of $B$ associated to the eigenvalues $\lambda, \mu,$ and $\nu$, respectively.  
Let $v_\lambda,v_\mu,$ and $v_\nu$ denote corresponding right eigenvectors.  

Consider the torsion-free abelian group $H = \bigcup_{n=0}^\infty B^{-n}(\Z^3)$.  
Any element $g\in H$ can be written as a combination $g = \alpha_\lambda v_\lambda + \alpha_\mu v_\mu + \alpha_\nu v_\nu$, but in this case the coefficients will lie in the splitting field $K$ of $\chi_A(x)$ over $\Q_3$.  
As was the case in Example \ref{EX:favourite}, $g$ is divisible by high powers of $3$ if and only if $\alpha_\mu$ and $\alpha_\nu$ are small in the extension of the $3$-adic absolute value to $K$.  
Thus the only group homomorphisms from $H$ into $\Z_3$ are the combinations of $w_\mu$ and $w_\nu$.  

Using the facts that $\lambda\mu\nu = -3$, $\lambda\mu + \mu\nu + \nu\lambda = 1$, and $\lambda + \mu + \nu = -1$, we can calculate 
\begin{align*}
w_\mu + w_\nu & = (2,-1-\lambda,-1-\lambda^2)\in (\Z_3^3)^* \\
\text{and } \nu w_\mu + \mu w_\nu & = (-1-\lambda,2+2\lambda+2\lambda^2,2-\lambda-\lambda^2) \in (\Z_3^3)^*.  
\end{align*}
Any homomorphism $f: H \to \Z_3$ can be written as a $\Z_3$-combination of these two rows.  
(This follows from a rank argument:  the group $Hom (H,\Z_3)$ is a free $\Z_3$-module, its rank cannot exceed $3$, it does not contain any non-zero multiple of $w_\lambda$, and we have found two independent rows that it does contain.)  
\end{example}

\section{$p$-adic functionals}\label{SEC:states}  

The goal of this section is to study the group $Hom (G,\Z_p)$, of homomorphisms from a torsion-free abelian group $G$ into $\Z_p$, from a functional analysis point of view.  
To emphasize this point of view, these homomorphisms will be referred to as \defemph{$p$-adic functionals} on $G$.  
\begin{defn}\label{DEF:p-adic-functional}
Let $p$ be a prime and let $G$ be a torsion-free abelian group.  
Then a \defemph{$p$-adic functional} on $G$ is a group homomorphism $f : G\to \Z_p$.  
\end{defn}

The starting point of this work is Proposition \ref{PROP:all-states-extend}, which is an analogue of the Hahn--Banach Theorem for the setting of torsion-free abelian groups.  
Proposition \ref{PROP:all-states-extend} says that if $f$ is a $\Z_p$-valued group homomorphism on a subgroup $H$ of $G$ and $f$ is \defemph{$t$-contractive} for some \defemph{$p$-adic sublinear pseudometric} $t$ (Definition \ref{DEF:t-contractive}), then $f$ extends to a $p$-adic functional on $G$ that is also $t$-contractive.  
This allows us to establish the existence of $p$-adic functionals on $G$ with desired properties by first finding functionals with those properties defined on some subgroup and then extending them to all of $G$.  

A consequence of this theory is that the group of all $\Z_p$-valued homomorphisms on $G$ (which will be denoted by $G^{*p}$ to emphasize the functional analysis point of view) contains enough functionals to be interesting.  
Section \ref{SUBSEC:state-spaces} deals with this group, both as a topological space and as a torsion-free abelian group.  

The elements of $G$ are themselves $p$-adic functionals on $G^{*p}$, and there is a homomorphism from $G$ into the double dual group, which consists of $p$-adic functionals on the group of $p$-adic functionals on $G$.  
It will be shown in Section \ref{SEC:completions} that this double dual group $G^{*p*p}$ coincides exactly with the pro-$p$ completion of $G$.  
This result provides a means of describing the pro-$p$ completion of $G$, which is interesting in itself.  

$p$-adic functionals are a $\Z_p$-submodule of the set of $\Q_p$-valued group homomorphisms on $G$ that are contractive with respect to a certain pseudometric; let us describe this pseudometric now.  

\subsection{The $p$-adic pseudometric}\label{SUBSEC:pseudometric}

If $g$ is an element of a torsion-free abelian group $G$ and $n\in\N$, let us say that $g$ is \defemph{divisible by $n$} or \defemph{$n$-divisible in $G$} if there exists $h\in G$ such that $nh = g$.  
Let us say that $G$ is $n$-divisible if each element of $G$ is $n$-divisible.  

\begin{defn}\label{DEF:pseudometric}
Let $G$ be a torsion-free abelian group and $p$ be a prime.  
Define the \defemph{$p$-adic pseudometric} $d_p$ on $G$ by 
\[
d_p(g,h) = p^{-k},
\]
where $k$ is the non-negative integer that is maximal with the property that $g-h$ is divisible by $p^k$.  
If $g-h$ is divisible by unbounded powers of $p$, then $d_p(g,h) = 0$.  
\end{defn}

The reader is invited to verify that $d_p$ satisfies the definition of a pseudometric; indeed, it satisfies the strong triangle inequality: $d_p(g,h) \leq \max \{ d_p(g,z),d_p(z,h)\}$, where $g,h,$ and $z\in G$.  
$d_p$ also satisfies an absolute scalability condition: if $g,h\in G$ and $n\in\Z$, then $d_p(ng,nh) = |n|_pd_p(g,h)$.  

$d_p$ is a metric if and only if $0 \in G$ is the only element that is divisible by arbitrarily high powers of $p$; in this case, let us refer to it as the \defemph{$p$-adic metric}.  
$d_p$ is bounded: $d_p(g,h)\leq 1$ for all $g,h\in G$.  

The topology arising from the $p$-adic pseudometric is called the $p$-adic topology in \cite{K:abelian-groups}.  

Let us use the symbol $G_p$ to denote the subgroup of $G$ containing all elements that are divisible by arbitrarily high powers of $p$.  
It is not difficult to verify that $G_p$ is a pure subgroup.  
Given a subgroup $H$ of $G$, let $\overline{H}^p$ denote the closure of $H$ in the topology arising from $d_p$.  
Note that $\overline{\{ 0\}}^p = G_p$, so $\overline{H}^p$ contains $G_p$ for all subgroups $H$.  

The following lemma will be useful later, when proving Lemma \ref{LEM:quotient}, which is used to prove Theorem \ref{THM:pro-p}, one of the main theorems.  
It says that, if $H$ is a subgroup of $G$ such that any element of $G/H$ has order dividing $p^m$, then any element $g\in G$ that is not in $H$ must lie at a distance of at least $p^{-m}$ from $H$.  
In particular, this implies that if $[G:H] = p^m$, then $H = \overline{H}^p$.  
\begin{lemma}\label{LEM:subgroup-separation}
Let $p$ be a prime, let $G$ be a torsion-free abelian group with $p$-adic pseudometric $d_p$, let $H$ be a subgroup of $G$, and let $m\in\N$.  
Then the following are equivalent.  
\begin{enumerate}
\item  The order of any element $g+H\in G/H$ divides $p^m$.  
\item  If $g, g'\in G$ and $d_p(g,g') \leq p^{-m}$, then $g+H = g'+H$.  
\end{enumerate}
\end{lemma}
\begin{proof}
To prove that (1) implies (2), suppose that $g,g'\in G$ are such that $d_p(g,g') \leq p^{-m}$.  
This means there exists $h\in G$ such that $p^mh = g-g'$.  
But by (1), $p^mh\in H$, so 
\[
g + H = g - g' + g' + H = p^{m}h + g' + H = g' + H.
\]

To see that (2) implies (1), pick $g\in G$.  
Then 
\[
d_p(p^mg,0) = p^{-m}d_p(g,0) \leq p^{-m}, 
\]
so, by (2), $p^mg + H = 0 + H$, that is, $p^mg\in H$.  
\end{proof}

The following lemma establishes that the closure of a pure subgroup is pure.  

\begin{lemma}\label{LEM:pure-closure}
Let $p$ be a prime, let $G$ be a torsion-free abelian group, and let $H$ be a pure subgroup of $G$.  
Then $\overline{H}^p$ is also a pure subgroup.  
\end{lemma}
\begin{proof}
It is easy to verify that $\overline{H}^p$ is a subgroup: if $g,g'\in \overline{H}^p$, then given $\epsilon > 0$ we can find $h,h'\in H$ such that $d_p(g,h), d_p(g',h') < \epsilon$, in which case $d_p(g+g',h+h') \leq \max \{ d_p(g,h),d_p(g',h')\} < \epsilon$.  

To see that $\overline{H}^p$ is pure, suppose $ng\in \overline{H}^p$ for some $n\in \Z\backslash \{ 0\}$ and $g\in G$.  
This is equivalent to saying that $d_p(ng+H,0) = 0$ in the quotient group $G/H$, which is torsion-free as $H$ is pure.  
Thus $ng+H\in (G/H)_p$, which is pure in $G/H$, so $g+H\in (G/H)_p$, so $g\in \overline{H}^p$ as well.  
\end{proof}

Closed pure subgroups will play a role in what follows, so it will be useful to give a name to such subgroups.  
\begin{defn}\label{DEF:ideal-simple}
Let $p$ be a prime and let $G$ be a torsion-free abelian group.  
Let us use the term \defemph{$p$-ideal} to refer to any pure subgroup of $G$ that is closed in the topology arising from $d_p$.  
Let us say that $G$ is \defemph{$p$-simple} if $G_p$ is the only proper $p$-ideal of $G$.  
\end{defn}

The group $G$ in Example \ref{EX:favourite} is $3$-simple: it has many proper non-trivial pure subgroups, each of which is isomorphic to $\Z$, but the closure of any of them in the metric $d_3$ is all of $G$.  

Given a family $(J_\lambda)_{\lambda\in\Lambda}$ of $p$-ideals of $G$, the intersection $\bigcap_{\lambda\in\Lambda}J_\lambda$ is again a $p$-ideal.  
The sum of the $J_\lambda$s is not necessarily pure or closed, but $\overline{\langle \bigcup_{\lambda\in\Lambda}J_\lambda \rangle_*}^p$ is a $p$-ideal, and is clearly the smallest one that contains each $J_\lambda$.  
Thus the set of $p$-ideals of $G$ forms a complete lattice where the partial order is given by inclusion.  
The minimal $p$-ideals---namely, the elements of this lattice that have nothing except $G_p$ below them---are of particular interest.  

\begin{defn}\label{DEF:minimal-ideal}
Let $p$ be a prime and let $G$ be a torsion-free abelian group.  
A $p$-ideal $J$ of $G$ is called \defemph{minimal} if it properly contains $G_p$, but does not properly contain any other $p$-ideal of $G$.  
\end{defn}

Lemma \ref{LEM:singly-generated-ideal}, below, establishes that a singly-generated $p$-ideal is minimal.  
The proof of this fact uses the following lemma, which will also be used later to show that a $p$-ideal can be separated from an element not in its closure using a $p$-adic functional (Corollary \ref{COR:separating-state}).  

\begin{lemma}\label{LEM:closure}
Let $p$ be a prime, let $G$ be a torsion-free abelian group, let $g\in G$, and let $H$ be a pure subgroup of $G$.  
If the distances $\{ d_p(g,h) : h\in H\}$ are bounded below by $\epsilon > 0$, then the values $\{\frac{1}{|n|_p}d_p(ng,h) : h\in H, n\in \Z\backslash \{ 0\}\}$ are also bounded below by $\epsilon$.  
\end{lemma}
\begin{proof}
Suppose otherwise, and let $n\in \N$ be minimal with the property that $\frac{1}{|n|_p}d_p(ng,h) < \epsilon$ for some $h\in H$.  
Let us show first that $n$ is not divisible by $p$.  

Suppose for a contradiction that $p$ that divides $n$---say $n = pn'$ for $n'\in \N$.  
The assumption that $d_p(ng,h) < \epsilon|n|_p < 1$ implies that $ng-h$ is $p$-divisible.  
Combined with the obvious fact that $ng$ is $p$-divisible, this implies that $h$ is $p$-divisible---say $h = ph'$ for some $h'\in H$.  
But then $\frac{1}{|n'|_p}d_p(n'g,h') = \frac{1}{|n|_p}d_p(ng,h) < \epsilon$, contradicting the minimality of $n$.  
Thus $n$ is not divisible by $p$.  

Consider the element $g+H\in G/H$.  
As $H$ is pure, $G/H$ is torsion-free.  
$d_p(ng,h) < \epsilon$ in $G$, so $d_p(ng + H, 0) < \epsilon$ in $G/H$.  
But $p$ does not divide $n$, so $d_p(g+H,0) = |n|_p d_p(g+H,0) = $ $d_p(ng+H,0) < \epsilon$.  
But the statement that $d_p(g+H,0)< \epsilon$ in $G/H$ implies that $d_p(g,h) < \epsilon$ in $G$ for some $h\in H$, which is a contradiction.  
\end{proof}

\begin{lemma}\label{LEM:singly-generated-ideal}
Let $p$ be a prime, let $G$ be a torsion-free abelian group, and suppose that $g\in G\backslash G_p$.  
Then $\overline{\langle g\rangle_*}^p$ is a minimal $p$-ideal.  
\end{lemma}
\begin{proof}
$\overline{\langle g\rangle_*}^p$ is certainly a $p$-ideal; to see that it is minimal, let us suppose that it properly contains another $p$-ideal $H$ and show that $H = G_p$.  
Choose $h\in H$ and consider the $p$-ideal $\overline{\langle h\rangle_*}^p$.  
By assumption, $g$ is not in this $p$-ideal, so there exists $\epsilon > 0$ such that $d_p(g,h')>\epsilon$ for all $h'\in \langle h\rangle_*$.  
By Lemma \ref{LEM:closure}, the values $\{ \frac{1}{|n|_p}d_p(ng,h') : n\in \Z\backslash \{ 0\}, h'\in \langle h\rangle_*\}$ are also bounded below by $\epsilon$.  

But the hypothesis that $h\in \overline{\langle g\rangle_*}^p$ means that, for all $k\in \N$, there exists $g_k\in \langle g\rangle_*$ such that $d_p(g_k,h)<p^{-k}$.  
There exist relatively prime $n_k\in\Z$ and  $m_k\in\Z\backslash \{ 0\}$ such that $m_kg_k = n_kg$.  
Thus $d_p(n_kg,m_kh) < p^{-k}$.  
Combined with the fact that $\frac{1}{|n_k|_p}d_p(n_kg,m_kh) > \epsilon$, this implies that $\frac{1}{|n_k|_p}p^{-k}> \epsilon$, that is, $|n_k|_p < p^{-k}/\epsilon$.  
Thus, given $\delta > 0$, the value $|n_k|_p$ can be made less than $\delta$ by choosing large enough $k$.  
So in particular, for large enough $k$, $p\mid n_k$, and hence $p\nmid m_k$.  
But then 
\begin{align*}
d_p(h,0) & = |m_k|_pd_p(h,0) \\
 & = d_p(m_kh,0) \\
 & \leq \max \{ d_p(n_kg,m_kh),d_p(n_kg,0)\} \\
 & \leq \max \{ p^{-k},\delta\}.  
\end{align*}
Thus $d_p(h,0) = 0$, so $h\in G_p$.  
\end{proof}
An immediate consequence of Lemma \ref{LEM:singly-generated-ideal} is the following.  
\begin{cor}\label{COR:minimal-ideals}
Let $p$ be a prime, let $G$ be a torsion-free abelian group, and let $H$ be a $p$-ideal of $G$.  
Then $H$ is minimal if and only if it is singly generated as a $p$-ideal and not equal to $G_p$.  
\end{cor}

\subsection{A Hahn--Banach Theorem for $p$-adic functionals}\label{SUBSEC:functionals}

Let $p$ be a prime, let $G$ be a torsion-free abelian group with $p$-adic pseudometric $d_p$, and let $f:G\to \Z_p$ be a group homomorphism (i.e., a $p$-adic functional).  
Then $f$ is automatically contractive (and hence continuous) with respect to the $p$-adic pseudometric $d_p$ on $G$, meaning that $|f(g_1)-f(g_2)|_p\leq d_p(g_1,g_2)$ for all $g_1,g_2\in G$.  
To see this, it is enough to observe that, if $g\in G$ is $p^n$-divisible---say $g = p^nh$ for some $h\in G$---then $|f(g)|_p = |p^nf(h)|_p = p^{-n}|f(h)|_p\leq p^{-n}$, as $|f(h)|_p\leq 1$.  

If instead $f$ is only defined on a subgroup $H$ of $G$, then the above argument shows that, in order for there to exist an extension of $f$ to all of $G$, it is necessary that $f$ be contractive with respect to $d_p$.  
It turns out that this condition is also sufficient (Proposition \ref{PROP:all-states-extend}).  
But we can say more: if $f: H\to \Z_p$ is contractive with respect to some \defemph{$p$-adic sublinear pseudometric} $t$ (Definition \ref{DEF:t-contractive}), then $f$ extends to a $p$-adic functional on $G$ that is also contractive with respect to $t$.  

\begin{defn}\label{DEF:t-contractive}
Let $p$ be a prime and let $G$ be a torsion-free abelian group with $p$-adic pseudometric $d_p$.  
Let $t : G\times G \to \R_{\geq 0}$ be a pseudometric.  
Let us say that $t$ is a \defemph{$p$-adic sublinear pseudometric} on $G$ if it satisfies the following additional conditions:  
\begin{enumerate}
\item  $t(g_1,g_2) = t(g_1+h,g_2+h)$ for all $g_1,g_2,h\in G$ ($t$ is translation invariant); 
\item  $|n|_pt(g_1,g_2) = t(ng_1,ng_2)$ for all $g_1,g_2\in G$, $n\in \Z$ ($t$ is absolutely scalable); 
\item  $t(g_1,g_2) \leq d_p(g_1,g_2)$ for all $g_1,g_2\in G$ ($t$ is dominated by $d_p$); and 
\item  $t(g_1,g_2) \leq \max\{ t(g_1,g_3),t(g_2,g_3)\}$ for all $g_1,g_2,g_3\in G$ ($t$ satisfies the strong triangle inequality).  
\end{enumerate}

Let $H$ be a subgroup of $G$, and suppose that $f : H\to \Z_p$ is a group homomorphism that is contractive with respect to the $p$-adic sublinear pseudometric $t$, meaning that $|f(h_1)-f(h_2)|_p \leq t(h_1,h_2)$ for all $h_1,h_2\in H$.  
Let us refer to such a homomorphism $f$ as a \defemph{$t$-contractive functional} on $H$.  
\end{defn}

Note in particular that $d_p$ itself is a $p$-adic sublinear pseudometric, and,  if $t$ is another $p$-adic sublinear pseudometric, then any $t$-contractive functional is also a $d_p$-contractive functional.  
Most of the results in this section are statements about the existence of $p$-adic functionals; these functionals are extensions of $t$-contractive functionals for some $p$-adic sublinear pseudometric $t$.  
For most of the proofs, $t$ can be taken to be $d_p$; however, Proposition \ref{PROP:separate-from-subgroup} establishes the existence of $p$-adic functionals with very specific properties, and for this it will be necessary to use a more restrictive $t$ to get sharper information.  

Note that, if $H$ is a subgroup of $G$ and $d_p$ is the $p$-adic pseudometric on $G$, then the $p$-adic pseudometric on $H$ is bounded below by the restriction of $d_p$ to $H$, and this bound may be strict.  
Thus to say that $f$ is a $d_p$-contractive functional on $H$ is stronger than saying that $f$ is a $p$-adic functional on $H$.  
Of course, if $H$ is a pure subgroup of $G$, then the $p$-adic pseudometric on $H$ coincides with the restriction to $H$ of that on $G$.  

It is clear that the trivial homomorphism is a $p$-adic functional, and, if $G$ is $p$-divisible, then this is the only one.  
It is natural to ask whether there exist any non-trivial $p$-adic functionals on $G$ if $G$ is not $p$-divisible.  
Corollary \ref{COR:states-exist} shows that, if $G$ is not $p$-divisible, then non-zero $p$-adic functionals on $G$ do exist.  

The following is the key lemma that makes everything else work.  

\begin{lemma}\label{LEM:extend}
Let $p$ be a prime, let $G$ be a torsion-free abelian group, let $t$ be a $p$-adic sublinear pseudometric on $G$, let $H$ be a subgroup of $G$, and let $g\in G$.  
Suppose there exists a $t$-contractive functional $f: H\to \Z_p$, and define, for each $h\in H$ and $n\in \Z\backslash \{ 0\}$, the closed subset $B(h,n) := \{ \alpha \in\Q_p : |n|_p|\alpha-\frac{1}{n}f(h)|_p \leq t(ng,h)\}$.  
Then
\begin{enumerate}
\item  $\bigcap_{h\in H,n\in\Z\backslash \{ 0\}} B(h,n)$ is non-empty.  
\item  If $f$ extends to a $t$-contractive functional $f'$ on $H + \Z g$, then $f'(g)\in \bigcap_{h\in H,n\in\Z\backslash \{ 0\}} B(h,n)$.  
\item  If $\beta \in \bigcap_{h\in H,n\in\Z\backslash \{ 0\}}B(h,n)$, then $f$ extends to a $t$-contractive functional $f'$ on $H + \Z g$ such that $f'(g) = \beta$.  
\end{enumerate}
\end{lemma}
\begin{proof}
To prove statement (1), note that $B(0,1)$ consists of all $\alpha\in \Q_p$ such that $|\alpha|_p\leq t(g,0) \leq d_p(g,0) \leq 1$, so $B(0,1)\subseteq \Z_p$.  
Then rewrite the intersection:  
\[
\bigcap_{h\in H,n\in\Z\backslash\{ 0\}}B(h,n) = \bigcap_{h\in H,n\in\Z\backslash \{ 0\}} (B(0,1)\cap B(h,n)), 
\]
the latter of which is an intersection of closed subsets of $\Z_p$.  
$\Z_p$ is compact, so the collection of all closed subsets of the form $B(0,1)\cap B(h,n)$ has non-empty intersection if and only if it has the finite intersection property.  
But $\bigcap_{i=1}^k(B(0,1)\cap B(h_i,n_i)) = B(0,1)\cap B(h_1,n_1)\cap \cdots \cap B(h_k,n_k)$, so it suffices to check that the collection of all subsets of the form $B(h,n)$ has the finite intersection property.  

Because $|\cdot |_p$ satisfies the strong triangle inequality, if two closed disks in $\Q_p$ have non-empty intersection, then one is contained in the other.  
Thus if $B(h_1,n_1)\cap B(h_2,n_2)$ is non-empty, it equals either $B(h_1,n_1)$ or $B(h_2,n_2)$.  
So to show that a finite intersection of closed disks is non-empty, it suffices to show that all pairwise intersections are non-empty.  
So let us verify to that any pair of subsets $B(h_1,n_1)$ and $B(h_2,n_2)$ with $h_1,h_2\in H$, $n_1,n_2\in \Z\backslash \{ 0\}$ has non-empty intersection.  

By the strong triangle inequality, 
\[
t(n_2h_1,n_1h_2) \leq \max \{ t(n_1n_2g,n_2h_1),t(n_1n_2g,n_1h_2)\}.  
\]
Suppose without loss of generality that $t(n_1n_2g,n_2h_1)$ is the maximum of these two distances.  
Then 
\begin{align*}
|n_1n_2|_p |\frac{1}{n_1}f(h_1)-\frac{1}{n_2}f(h_2)|_p
 & = |n_2f(h_1)-n_1f(h_2)|_p \\
 & = |f(n_2h_1)-f(n_1h_2)|_p \\
 & \leq t(n_2h_1,n_1h_2) \\
 & \leq t(n_1n_2g,n_2h_1) = |n_2|_pt(n_1g,h_1),  
\end{align*}
where the last equality comes from the absolute scalability of $t$.  

Thus $|n_1|_p|\frac{1}{n_1}f(h_1)-\frac{1}{n_2}f(h_2)|_p $ $\leq t(n_1g,h_1)$, so $\frac{1}{n_2}f(h_2)\in B(h_1,n_1)$.  
It is clearly in $B(h_2,n_2)$, so the intersection is non-empty.  

Now let us prove statement (2).  

For all $h\in H$ and $n\in\Z\backslash\{ 0\}$, $|n|_p|f'(g)-\frac{1}{n}f(h)|_p $ $ = |f'(ng) - f'(h)|_p$, which is bounded above by $t(ng,h)$ because $f'$ is a $t$-contraction.  
Thus $f'(g)\in B(h,n)$.  

Now let us prove statement (3).  

Define $f': H + \Z g \to \Q_p$ as follows.  
Given $g'\in H + \Z g$, write $g' = h + ng$ for some $h\in H$ and $n\in \Z$.  
Then let $f'(g') = f(h) + n\beta\in \Q_p$.  

$f'$ is well-defined.  
To see this, suppose $g' = h_1+n_1g$ and $g' = h_2 + n_2g$; then $(n_1-n_2)g' - (h_2-h_1) = 0$.  
If $n_1 = n_2$, then $h_1 = h_2$, in which case $f(h_1)+n_1\beta = f(h_2)+n_2\beta$.  
If $n_2\neq n_2$, then $d_p((n_1-n_2)g,h_2-h_1) = 0$, so $B(h_2-h_1,n_1-n_2)$ is the singleton $\{ \frac{1}{n_1-n_2}f(h_2-h_1)\}$, in which case $\beta = \frac{1}{n_1-n_2}f(h_2-h_1)$ and $f(h_1) + n_1\beta = f(h_2)+n_2\beta$.  

It is easy to verify that $f'$ is a group homomorphism.  

Given that $f'$ is a group homomorphism and $t$ is translation invariant, to see that $f'$ is a $t$-contractive functional it is enough to check that $|f(g')|_p\leq t(g',0)$ for all $g'\in H+\Z g$.  
If $g'\in H$, this follows immediately from the fact that $f$ is a $t$-contractive functional, so we may suppose that $g'\notin H$.  
This means that $g' = h + ng$ for some $h\in H$ and $n\in \Z\backslash \{ 0\}$.  
Then 
\begin{align*}
|f'(g')|_p & = |f(h)+n\beta|_p \\
 & = |n|_p|\beta - \frac{1}{n}f(-h)|_p \\
 & \leq t(ng,-h) = t(h+ng,0),
\end{align*}
where the inequality on the last line comes from the fact that $\beta\in B(-h,n)$.  
\end{proof}

Lemma \ref{LEM:extend} allows us to extend a $t$-contractive functional on a subgroup of $G$ to one on a larger subgroup, but to show that non-trivial $p$-adic functionals exist, we need to start with a non-trivial $t$-contractive functional on some subgroup of $G$.  
The natural place to begin is with a cyclic subgroup.  

The following corollary, which deals with functionals on cyclic subgroups of $G$, follows from Lemma \ref{LEM:extend} by taking $H = \{ 0\}$ and $t = d_p$.  
\begin{cor}\label{COR:rank-one-seed}
Let $p$ be a prime, let $G$ be a torsion-free abelian group with $p$-adic pseudometric $d_p$, and let $g\in G$.  
Then, given any $\beta\in\Z_p$ for which $|\beta|_p \leq d_p(g,0)$, there exists a $d_p$-contractive functional $f:\langle g\rangle\to \Z_p$ such that $f(g) = \beta$.  
\end{cor}

The next proposition is an analogue of the Hahn--Banach Theorem for the setting of torsion-free abelian groups.  
\begin{prop}\label{PROP:all-states-extend}
Let $p$ be a prime and let $G$ be a torsion-free abelian group with $p$-adic sublinear pseudometric $t$.  
Suppose that $H$ is a subgroup of $G$ and $f$ is a $t$-contractive functional on $H$.  
Then there exists a $t$-contractive $p$-adic functional $f'$ on $G$ that extends $f$.  
\end{prop}
\begin{proof}
Consider the set $\mathcal{K}$ of pairs $(f',H')$, where $H'$ is a subgroup of $G$ that contains $H$ and $f'$ is a $t$-contractive functional on $H'$ that extends $f$.  
$\mathcal{K}$ is non-empty because it contains $(f,H)$.  
Give $\mathcal{K}$ a partial ordering by saying that $(f_1',H_1')\leq (f_2',H_2')$ if $H_1'\subseteq H_2'$ and $f_2'$ extends $f_1'$.  
Any chain $((f_\lambda',H_\lambda'))_{\lambda\in\Lambda}$ in $\mathcal{K}$ is bounded above---an upper bound is $(F',\bigcup_{\lambda\in\Lambda}H_\lambda')$, where $F'$ is defined by $F'(h) = f_\lambda'(h)$ for any $\lambda$ for which $h\in H_\lambda'$.  

Thus by Zorn's Lemma there exists a maximal pair $(f_{\max},H_{\max})$.  
Suppose for a contradiction that $H_{\max}\neq G$.  
Then there exists $g'\in G\backslash H_{\max}$.  
By parts (1) and (3) of Lemma \ref{LEM:extend}, $f_{\max}$ extends to a $t$-contractive functional $f_{\max}'$ on the subgroup $H+\Z g'$.  
$H+\Z g'$ contains $H_{\max}$, and hence contains $H$, and $f_{\max}'$ extends $f_{\max}$, and hence extends $f$.  
Thus the pair $(f_{\max}',H+\Z g')$ is in $\mathcal{K}$, and is strictly greater than $(f_{\max},H_{\max})$, contradicting the maximality of that pair.  
Hence $H_{\max} = G$ after all.  
\end{proof}

Proposition \ref{PROP:all-states-extend} can be combined with Corollary \ref{COR:rank-one-seed} and Lemma \ref{LEM:closure} to produce the following corollaries, which say that any $g\in G$ realizes its maximum possible value on some functional, and $p$-adic functionals can separate a pure subgroup of $G$ from an element $g$ not in its closure.  
\begin{cor}\label{COR:state-hits-distance}
Let $p$ be a prime, let $G$ be a torsion-free abelian group with $p$-adic pseudometric $d_p$, and let $g\in G$.  
Then there exists a $p$-adic functional $f$ on $G$ such that $|f(g)|_p = d_p(g,0)$.  
\end{cor}
\begin{proof}
By Corollary \ref{COR:rank-one-seed}, there exists a $d_p$-contractive functional $f$ on $\langle g\rangle$ such that $|f(g)|_p = d_p(g,0)$.  
By Proposition \ref{PROP:all-states-extend}, $f$ extends to a $p$-adic functional on $G$.  
\end{proof}

\begin{cor}\label{COR:states-exist}
Let $p$ be a prime and let $G$ be a torsion-free abelian group that is not $p$-divisible.  
Then there exists a non-zero $p$-adic functional on $G$.  
\end{cor}

\begin{cor}\label{COR:separating-state}
Let $p$ be a prime, let $G$ be a torsion-free abelian group with pure subgroup $H$, and suppose that $g$ is an element of $G$ that is not in $\overline{H}^p$.  
Then there exists a $p$-adic functional $f$ on $G$ such that $f(h) = 0$ for all $h\in H$, but $f(g)\neq 0$.  
\end{cor}
\begin{proof}
Let $d_p$ denote the $p$-adic pseudometric on $G$.  
The zero homomorphism $f' : H\to \Z_p$ is a $d_p$-contractive functional.  

To say that $g$ is not in $\overline{H}^p$ means that there is some $\epsilon > 0$ such that $d_p(g,h) \geq  \epsilon$ for all $h\in H$.  
By Lemma \ref{LEM:closure}, this means that the values $\{ \frac{1}{|n|_p}d_p(ng,h) : h\in H, n\in \Z\backslash \{ 0\}\}$ are bounded below by $\epsilon$.  
But then, by Lemma \ref{LEM:extend} part (3), for any $\beta$ in the ideal $\{ \alpha\in\Z_p : |\alpha|_p\leq \epsilon\}$ of $\Z_p$, $f'$ extends to a $d_p$-contractive functional $f''$ on $H+\Z g$ such that $f''(g) = \beta$.  
In particular, it is possible to find a non-zero $\beta$ with this property.  
Then by Proposition \ref{PROP:all-states-extend}, $f''$ extends to the required $p$-adic functional $f$ on $G$.  
\end{proof}

The following proposition will be useful in Section \ref{SEC:completions}, when we consider finite-index subgroups of a torsion-free abelian group $G$.  

\begin{prop}\label{PROP:separate-from-subgroup}
Let $p$ be a prime, let $G$ be a torsion-free abelian group, and let $H$ be a subgroup of $G$.  
Suppose $m\in \N$ is such that the order of any element $g+H\in G/H$ divides $p^m$.  
Then for any $g\in G\backslash H$ there exists a $p$-adic functional $f$ on $G$ with the property that $|f(g)|_p >p^{-m} \geq |f(h)|_p$ for all $h\in H$.  
\end{prop}
\begin{proof}
Let $d_p$ denote the $p$-adic pseudometric on $G$, and let $d_p^H$ denote the $p$-adic pseudometric on $H$ (which might be different from the restriction of $d_p$ to $H$).  
For $g\in G$, define 
\begin{align*}
v : G & \to \N\cup \{ 0\} \\
g & \mapsto \min \{ k\in \N\cup\{ 0\} : p^kg\in H\}.  
\end{align*}
Then define a function $t : G\times G \to \R_{\geq 0}$ as follows:  
\begin{align*}
t(g_1,g_2) & = \left\{ \begin{array}{rl} p^{-m}d_p^H(g_1,g_2) & \text{ if } g_1-g_2\in H \\
p^{-m}p^{v(g_1-g_2)} & \text{ otherwise.} \end{array}\right. 
\end{align*}

Let us show that $t$ is a $p$-adic sublinear pseudometric.  

Choose elements $g_1,g_2$, and $g_3\in G$.  
The proof that $t$ satisfies the strong triangle inequality can be broken into cases, depending upon which of the differences $g_1-g_2$, $g_1-g_3$, $g_3-g_2$ lie in $H$.  
If any two of these are in $H$, then so is the third, so the possibilities are (1) all are in $H$; (2) none is in $H$; (3) $g_1-g_2\in H$, $g_1-g_3,g_3-g_2\notin H$; and (4) without loss of generality, $g_1-g_2,g_3-g_2\notin H$, $g_1-g_3\in H$.  

{\bf Case (1):}  $g_1-g_2,g_1-g_3,g_3-g_2\in H$.  
Then  
\begin{align*}
t(g_1,g_2) & = p^{-m}d_p^H(g_1,g_2) \\ 
 & \leq p^{-m}\max \{ d_p^H(g_1,g_3),d_p^H(g_3,g_2)\} \\ 
 & = \max\{ t(g_1,g_3),t(g_3,g_2)\}.  
\end{align*}

{\bf Case (2):}  $g_1-g_2,g_1-g_3,g_3-g_2\notin H$.  
Let $v_0 = \max\{ v(g_1-g_3),v(g_3-g_2)\}$.  
Then $p^{v_0}(g_1-g_3), p^{v_0}(g_3-g_2)\in H$, and 
\[
p^{v_0} (g_1-g_2) = p^{v_0} (g_1-g_3)  + p^{v_0} (g_3-g_2) \in H,
\]
so $v(g_1-g_2)\leq v_0 = \max\{ v(g_1-g_3),v(g_3-g_2)\}$, which means
\begin{align*}
t(g_1,g_2) = p^{-m}p^{v(g_1-g_2)} & \leq \max\{p^{-m}p^{v(g_1-g_3)},p^{-m}p^{v(g_3-g_2)}\} \\
 & = \max\{ t(g_1,g_3),t(g_3,g_2)\}.  
\end{align*}

{\bf Case (3):}  $g_1-g_2\in H$, $g_1-g_3,g_3-g_2\notin H$.  
In this case $t(g_1,g_3),$ $t(g_3,g_2) \geq p^{-m}$, so 
\[
t(g_1,g_2) = p^{-m}d_p^H(g_1,g_2)  \leq p^{-m}  \leq \max\{ t(g_1,g_3),t(g_3,g_2)\}.  
\]

{\bf Case (4):}  $g_1-g_2,g_3-g_2\notin H$, $g_1-g_3\in H$.  
In this case $p^k (g_1-g_2) \in H$ if and only if $p^k(g_3-g_2)\in H$, so 
\[
t(g_1,g_2) = p^{-m}p^{v(g_1-g_2)} = p^{-m}p^{v(g_3-g_2)} = t(g_3,g_2).
\]

Thus $t$ satisfies the strong triangle inequality.  
Combined with the obvious facts that $t$ is symmetric and $t(g,g) = 0$ for all $g\in G$, this implies that $t$ is a pseudometric.  

$t$ is clearly translation invariant, as the formulas $v(g_1-g_2)$ and $d_p^H(g_1,g_2)$ involve taking the difference of the group elements $g_1$ and $g_2$.  

To see that $t$ is absolutely scalable, pick $n\in \Z$ and $g_1,g_2\in G$.  
If $(g_1-g_2)\in H$, then $n(g_1-g_2)\in H$, and 
\begin{align*}
t(ng_1,ng_2) & = p^{-m}d_p^H(ng_1,ng_2) \\
 & = |n|_pp^{-m}d_p^H(g_1,g_2) \\
 & = |n|_pt(g_1,g_2).  
\end{align*}
If $n(g_1-g_2)\notin H$, then $g_1-g_2\notin H$, and 
\begin{align*}
t(ng_1,ng_2) & = p^{-m}p^{v(n(g_1-g_2))} \\
 & = p^{-m}p^{v(g_1-g_2)-v_p(n)} \\
 & = p^{-m}p^{v(g_1-g_2)}|n|_p \\
 & = |n|_pt(g_1,g_2), 
\end{align*}
where $v_p$ is the $p$-adic valuation on $\Q$.  
If $n(g_1-g_2)\in H$ but $g_1-g_2\notin H$, then $p^{v(g_1-g_2)} \mid n$; say $n = p^{v(g_1-g_2)}n'$, and
\begin{align*}
t(ng_1,ng_2) & = p^{-m}d_p^H(ng_1,ng_2) \\
 & = p^{-m}|n'|_pd_p^H(p^{v(g_1-g_2)}g_1,p^{v(g_1-g_2)}g_2) \\
 & = p^{-m}p^{-v_p(n')}\cdot 1 \\
 & = p^{-m}p^{v(g_1-g_2)}p^{-v(g_1-g_2)-v_p(n')} \\
 & = t(g_1,g_2)p^{-v_p(n)} \\
 & = |n|_pt(g_1,g_2).
\end{align*}

Now let us check that $t$ is dominated by $d_p$.  
Suppose that $g_1,g_2\in G$ and $g_1-g_2\in H$.  
Suppose $g_1-g_2$ is $p^n$-divisible in $G$ for some $n\geq 0$.  
If $n < m$, then $t(g_1,g_2) \leq p^{-m} < p^{-n}$.  
If $n\geq m$, then there exists $g\in G$ such that $g_1-g_2 = p^ng = p^{n-m}p^mg \in p^{n-m}H$, so $g_1-g_2$ is $p^{n-m}$-divisible in $H$, and $t(g_1,g_2) = p^{-m}d_p^H(g_1,g_2) \leq p^{-m}p^{-(n-m)} = p^{-n}$.  
As $n$ was arbitrary with the property that $g_1-g_2$ was $p^n$-divisible in $G$, it follows that $t(g_1,g_2)\leq d_p(g_1,g_2)$ if $g_1-g_2\in H$.  

If $g_1-g_2\notin H$ and $p^n$ divides $g_1-g_2$, then $n < m$ by the hypothesis on $m$, and $p^{m-n}(g_1-g_2)\in H$.  
Thus $v(g_1-g_2)\leq m-n$, so $t(g_1,g_2) = p^{-m}p^{v(g_1-g_2)} \leq p^{-n}$.  
Thus $t(g_1,g_2)\leq d_p(g_1,g_2)$ if $g_1-g_2\notin H$, so $t$ is dominated by $d_p$.  

This completes the proof that $t$ is a $p$-adic sublinear pseudometric.  

Now for any $g\notin H$, the definition of $t$ implies that $p^{-m} < t(g,0)$.  
Choose a group homomorphism $f' : \langle g\rangle \to \Z_p$ with the property that $|f'(g)|_p = t(g,0)$.  
Then $f'$ is $t$-contractive, so by Proposition \ref{PROP:all-states-extend}, it extends to a $t$-contractive $p$-adic functional $f$ on $G$.  
As $f$ is $t$-contractive, $|f(h)|_p\leq t(h,0) \leq p^{-m} < t(g,0) = |f(g)|_p$ for all $h\in H$.  
\end{proof}

\subsection{Spaces of functionals}\label{SUBSEC:state-spaces}  

\begin{defn}\label{DEF:state-space}
Let $p$ be a prime and let $G$ be a torsion-free abelian group.  
Let us use the symbol $G^{*p}$ to denote $Hom(G,\Z_p)$, the set of $p$-adic functionals on $G$.  
\end{defn}

$G^{*p}$ is a torsion-free abelian group, and as it is a subgroup of $\Z_p^G$, it contains no non-zero elements that are divisible by arbitrarily high powers of $p$, so the $p$-adic pseudometric on $G^{*p}$ is in fact a metric.  

In what follows, it will be useful to view $G^{*p}$ as a topological space.  
It turns out that the most useful topology is not the metric topology arising from the $p$-adic metric on $G^{*p}$, but rather the topology that $G^{*p}$ inherits from $\Z_p^{G}$.  

\begin{defn}\label{DEF:relative-top}
Let $p$ be a prime and let $G$ be a torsion-free abelian group.  
Let us refer to the topology that $G^{*p}$ inherits as a subspace of the product space $\Z_p^G$ as the \defemph{relative topology} on $G^{*p}$.  
\end{defn}
$\Z_p^G$ is a product of compact spaces, and hence is compact by Tychonoff's Theorem.  
It is easy to verify that $G^{*p}$ is closed in the relative topology, and hence is compact.  

$\Z_p$ is a principal ideal domain, and $G^{*p}$ is a submodule of the $\Z_p$-module $\Z_p^G$, which is free, so $G^{*p}$ is free as well \cite[Theorem 9.8]{R:algebra}.  
Thus is makes sense to speak of a $\Z_p$-basis and the $\Z_p$-rank of $G^{*p}$.  

The metric topology on $G^{*p}$ is finer than the relative topology in general, and if the $\Z_p$-rank of $G^{*p}$ is finite, then the two topologies agree, but this need not be true if $G^{*p}$ has infinite rank, as the following example illustrates.  

\begin{example}\label{EX:infinite-rank-independence}
Let $p$ be a prime, let $G = \bigoplus_{n\in\N}\Z$ and, for $n\in\N$, let $e_n$ denote the $n$th standard basis element of $G$.  
For $m\in\N$, let $f_m$ denote the $p$-adic functional defined by $f_m(e_n) = \delta_{m,n}$, and let $f_0$ denote the $p$-adic functional defined by $f_0(e_n) = 1$ for all $n\in\N$.  
Then the sequence $f_1,f_1+f_2,f_1+f_2+f_3,\ldots $ converges to $f_0$ in the relative topology, but not in the metric topology.  
\end{example}

Note that any $p$-adic functional on $G$ is entirely determined by its values on a maximal independent subset of $G$, which means that, if $G$ has finite rank, then the $\Z_p$-rank of $G^{*p}$ cannot exceed the rank of $G$---more precisely, it cannot exceed the rank of $G$ minus the rank of $G_p$.  

Proposition \ref{PROP:simple}, below, says that $G$ is $p$-simple but not $p$-divisible if and only if the $\Z_p$-rank of $G^{*p}$ is $1$.  
The proof of this fact uses the following easy lemma.  
\begin{lemma}\label{LEM:kernel-ideal}
Let $p$ be a prime, let $G$ be a torsion-free abelian group, and let $f$ be a $p$-adic functional on $G$.  
Then $\ker (f)$ is a $p$-ideal of $G$.  
\end{lemma}
\begin{proof}
$f$ is a group homomorphism into a torsion-free abelian group, so its kernel is a pure subgroup of $G$.  
$f$ is continuous with respect to the topology arising from the $p$-adic pseudometric on $G$, so its kernel is closed, as it is the pre-image of the closed subset $\{ 0\}$ of the metric space $\Z_p$.  
\end{proof}

\begin{prop}\label{PROP:simple}
Let $p$ be a prime and let $G$ be a torsion-free abelian group.  
Then $G$ is $p$-simple but not $p$-divisible if and only if, up to multiplication by a non-zero constant in $\Z_p$, there is a unique non-zero $p$-adic functional on $G$.  
\end{prop}
\begin{proof}
For the ``only if'' direction, suppose that $G$ is $p$-simple but not $p$-divisible, then pick $f_1, f_2 \in G^{*p}$, and suppose that $f_1$ is non-zero.  
Then there exists $g\in G$ such that $f_1(g)\neq 0$, which means in particular that $g\notin G_p$.  
Suppose without loss of generality that $|f_1(g)|_p\geq |f_2(g)|_p$.  
Then $f_0 := f_2 - \frac{f_2(g)}{f_1(g)}f_1\in G^{*p}$, and $g\in \ker(f_0)$.  
By Lemma \ref{LEM:kernel-ideal}, $\ker(f_0)$ is a $p$-ideal of $G$.  
$\ker(f_0)$ contains $g$, so is bigger than $G_p$; thus, as $G$ is $p$-simple, $\ker(f_0) = G$.  

Now let us prove the ``if'' direction.  

Note that, if $G$ is $p$-divisible, then it has no non-zero $p$-adic functionals.  
Thus suppose that $G$ is not $p$-divisible and not $p$-simple, and let us construct two $\Z_p$-independent $p$-adic functionals on $G$.  
Let $J$ be a proper $p$-ideal of $G$ that is strictly larger than $G_p$.  
As $J$ is proper, there exists $g\in G\backslash J$; by Corollary \ref{COR:separating-state}, there exists a $p$-adic functional $f_1$ with the property that $f_1(h) = 0$ for all $h\in J$ and $f_1(g)\neq 0$.  

As $J$ is larger than $G_p$, it is not $p$-divisible, so by Corollary \ref{COR:states-exist}, there exists a non-zero $p$-adic functional on $J$, which extends to a $p$-adic functional $f_2$ on $G$ by Proposition \ref{PROP:all-states-extend}.  
$f_1$ and $f_2$ are both non-zero, and neither is a multiple of the other because $f_1$ is $0$ on $J$ while $f_2$ is not.  
\end{proof}

The next lemma characterizes the $p$-adic metric on $G^{*p}$ in terms of the absolute values of the images of elements of $G$.  

\begin{lemma}\label{LEM:state-space-metric}
Let $p$ be a prime, let $G$ be a torsion-free abelian group with $p$-adic pseudometric $d_p$, and let $f$ be a $p$-adic functional on $G$.  
Then 
\[
d_p^{*p}(f,0) = \max \Big( \Big\{ \frac{|f(g)|_p}{d_p(g,0)} : g\in G, d_p(g,0)\neq 0\Big\}\cup \{ 0\}\Big),  
\]
where $d_p^{*p}$ denotes the $p$-adic metric on $G^{*p}$.  
\end{lemma}
\begin{proof}
Suppose that $f = p^k f'$ for some $f'\in G^{*p}$.  
Then, for any $g\in G$, 
\begin{align*}
|f(g)|_p & = |p^kf'(g)|_p = p^{-k}|f'(g)|_p \leq p^{-k}d_p(g,0).  
\end{align*}
Thus $\frac{|f(g)|_p}{d_p(g,0)} \leq d_p^{*p}(f,0)$ for all $g\in G$ with $d_p(g,0)\neq 0$.  

Now suppose that $|f(g)|_p\leq p^{-k} d_p(g,0)$ for all $g\in G$.  
Then define a function $f': G\to \Q_p$ by $f'(g) = \frac{1}{p^k}f(g)$ for $g\in G$.  
This is clearly a group homomorphism into $\Q_p$, and in fact it takes values in $\Z_p$ because $|f'(g)|_p = |\frac{1}{p^k}f(g)|_p = p^k|f(g)|_p \leq d_p(g,0)$, so in this case, $d_p^{*p}(f,0)\leq p^{-k}$.  
Thus $d_p^{*p}(f,0)\leq \max \{ \frac{|f(g)|_p}{d_p(g,0)} : g\in G, d_p(g,0)\neq 0\}$.  
\end{proof}

The following corollary is an immediate consequence of Lemma \ref{LEM:state-space-metric}.  
\begin{cor}\label{COR:state-max-attained}
Let $p$ be a prime, let $G$ be a torsion-free abelian group with $p$-adic pseudometric $d_p$, and suppose that $f$ is a non-zero $p$-adic functional on $G$.  
Then there exists $g\in G$ with $d_p(g,0) = 1$ such that $d_p^{*p}(f,0) = |f(g)|_p$, where $d_p^{*p}$ denotes the $p$-adic metric on $G^{*p}$.  
\end{cor}

The goal of the next sequence of results is to characterize the subsets of $G^{*p}$ that are sufficiently rich to contain, for each $g\in G$, a $p$-adic functional that achieves the maximum possible absolute value on $g$.  
This is the content of Proposition \ref{PROP:spanning-states}, which says that this is occurs for subsets of $G^{*p}$ that have $\Z_p$-spans that are dense in $G^{*p}$ in the relative topology.  
The proof of this result uses Lemma \ref{LEM:independence}, which in turn uses the following technical lemma.  

\begin{lemma}\label{LEM:independence-rank}
Let $p$ be a prime, let $G$ be a torsion-free abelian group, and let $F\subset G^{*p}$ and $X\subset G$ be finite subsets of $G^{*p}$ and $G$ respectively, with the property that $\{ f|_X : f\in F\} \subseteq \Z_p^X$ is $\Z_p$-independent.  
Then there exists $X'\subseteq X$ such that $X'$ and $F$ contain the same number of elements, and $\{ f|_{X'} : f\in F\}\subseteq \Z_p^{X'}$ is $\Z_p$-independent.  
\end{lemma}
\begin{proof}
Let $F = \{ f_1,\ldots,f_k\}$ and $X = \{ g_1, \ldots, g_l\}$.  
Certainly $l\geq k$, as $\{ f_i|_X : 1\leq i\leq k\}$ is independent.  

Consider the matrix $B\in M_{k\times l}(\Z_p)$, the $(i,j)$th entry of which is $f_i(g_j)$.  
View $B$ as a matrix over $\Q_p$; then if $l>k$ the Rank-Nullity Theorem says that there exists a non-zero vector $v = (v_1,\ldots, v_l)^t \in\Q_p^l$ such that $Bv = 0$.  
Suppose that the entry $v_l$ has maximal absolute value among all entries of $v$; then, by replacing $v$ with $v/v_l$, we may suppose that $v$ in fact has entries in $\Z_p$, and $v_l = 1$.  

Then $\{ f|_{X\backslash \{ g_l\}} : f\in F\}$ is $\Z_p$-independent.  
To see this, suppose that $\alpha_1f_1+\cdots + \alpha_kf_k$ is $0$ on $X\backslash \{ g_l\}$, where $\alpha_i\in\Z_p$ for $1\leq i\leq k$.  
Let us pick $\epsilon>0$ and show that $|\alpha_1f_1(g_l)+\cdots +\alpha_kf_k(g_l)|_p<\epsilon$.  

$\Z$ is dense in $\Z_p$, so we can pick an element $w = (w_1,\ldots,w_l)^t\in\Z^l$ such that $|v_j-w_j|_p<\epsilon$ for all $1\leq j\leq l$.  
We may insist further that $w_l=1$ because $v_l=1$.  
Then $|w_1f_i(g_l)+\cdots + w_{l-1}f_i(g_{l-1})+f_i(g_l)|_p<\epsilon$.  

Thus it is true that, for each $i\leq k$, 
\begin{align*}
f_i(g_l) & = \big( f_i(g_l) + \sum_{j=1}^{l-1}w_jf_i(g_j)\big) - \sum_{j=1}^{l-1}w_jf_i(g_j),\\
\intertext{which implies}
\sum_{i=1}^k \alpha_if_i(g_l) & = \sum_{i=1}^k\alpha_i \Big[ \big( f_i(g_l) + \sum_{j=1}^{l-1}w_jf_i(g_j)\big) - \sum_{j=1}^{l-1}w_jf_i(g_j) \Big] \\
 & = \sum_{i=1}^k\alpha_i \big( f_i(g_l) + \sum_{j=1}^{l-1}w_jf_i(g_j)\big) - \sum_{j=1}^{l-1}w_j \sum_{i=1}^k \alpha_i f_i(g_j) \\
 & = \sum_{i=1}^k\alpha_i \big( f_i(g_l) + \sum_{j=1}^{l-1}w_jf_i(g_j)\big) - \sum_{j=1}^{l-1}w_j  \cdot 0.
\end{align*}
Each term in the first sum has absolute value less than $|\alpha_i|\epsilon \leq  \epsilon$, so by the strong triangle inequality, the entire sum has absolute value less than $\epsilon$.  

Thus $\sum_{i=1}^k\alpha_if_i(g_l) = 0$, implying $\{ f|_X : f\in F\}$ is $\Z_p$-dependent, which is a contradiction.  
Thus $\{ f_{X\backslash \{ g_l\}} : f\in F\}$ is $\Z_p$-independent.  
We can apply the same argument to $X\backslash \{ g_l\}$ to show that either $X\backslash \{ g_l\}$ has the same number of elements as $F$, or else there exists a smaller set, the restriction of $F$ to which is $\Z_p$-independent.  
Proceeding in this fashion eventually produces a set with the same size as $F$.  
\end{proof}

The next lemma says that a functional $f\in G^{*p}$ can be separated from a subset $F\subseteq G^{*p}$ by an element of $G$ if and only if it is not in the relative-topology closure of the span of $F$.  

\begin{lemma}\label{LEM:independence}
Let $p$ be a prime, let $G$ be a torsion-free abelian group, let $F$ be a non-empty subset of $G^{*p}$, and let $f\in G^{*p}$.  
The following are equivalent:  
\begin{enumerate}
\item  There exists $g\in G$ such that $|f(g)|_p>|f'(g)|_p$ for all $f'\in F$.  
\item  $f$ is not in the relative-topology closure of the $\Z_p$-span of $F$.  
\end{enumerate}
\end{lemma}
\begin{proof}
To prove that (1) implies (2), suppose for a contradiction that there exists $g\in G$ satisfying condition (1), but $f$ lies in the relative-topology closure of the $\Z_p$-span of $F$.  
Let $M = \sup_{f'\in F} \{ |f'(g)|_p\}$.  
Then, as $\Z_p$ is a discrete valuation ring, this supremum is achieved on $F$; that is, $M = \max_{f'\in F} \{|f'(g)|_p\} < $ $|f(g)|_p$.  
Similarly, the fact that $\Z_p$ is a discrete valuation ring implies that the set $T_{|f(g)|_p} := \{ \beta\in\Z_p : |\beta |_p = |f(g)|_p\}$ is open in $\Z_p$.  

Consider the relative-topology open set $\prod_{g'\in G}U_{g'}$, where $U_{g'} = \Z_p$ if $g'\neq g$ and $U_g = T_{|f(g)|_p}$.  
$f$ is in this open set, so, by assumption, there exists some $\Z_p$ combination $\alpha_1f_1 + \cdots + \alpha_nf_n$ in this set, where $\alpha_i\in\Z_p$ and $f_1,\ldots , f_n\in F$.  
But then, by the strong triangle inequality, $|f(g)|_p = $ $|\alpha_1f_1(g) + \cdots + \alpha_nf_n(g)|_p \leq $ $\max_{1\leq i \leq n} \{ |\alpha_i|_p|f_i(g)|_p\} \leq M < |f(g)|_p$, which is a contradiction.  
Thus $f$ is not in the relative-topology closure of the $\Z_p$-span of $F$ after all.  

Now let us prove that (2) implies (1).  
The statement that $f$ is not in the relative-topology closure of the $\Z_p$-span of $F$ means that there exists a finite subset $X\subset G$ such that $f|_X$ is not in the $\Z_p$-span of $\{ f'|_X : f'\in F\}$.  
Let $f_1 = f|_X$, and let $\{ f_2,\ldots, f_k\}$ be a basis for the $\Z_p$-module spanned by $\{ f'|_X : f'\in F\}$, which is a submodule of the free $\Z_p$-module $\Z_p^X$, and hence is free itself.  

Then $\{ f_1,\ldots,f_k\}$ is independent over $\Z_p$.  
By Lemma \ref{LEM:independence-rank},  we may, in fact, suppose that $X$ contains exactly $k$ elements, say $X = \{ g_1,\ldots, g_k\}$.  

Let $A$ denote the $k\times k$ matrix, the $(i,j)$th entry of which is $f_i(g_j)\in\Z_p$.  
$\{ f_1|_X,\ldots , f_k|_X\}$ is $\Z_p$-independent, so $A$ has non-zero determinant.  

Let $\adj(A)$ denote the adjugate of $A$, that is, the transpose of the matrix of cofactors of $A$.  
Then the entries of $\adj(A)$ are in $\Z_p$, and $A\adj(A) = \det(A)I$, where $I$ is the $k\times k$ identity matrix.  

Let $e_1$ denote the first standard basis element of $\Z_p^k$.  
Let $v = (v_1,\ldots,v_k)^t =$ $\adj(A)e_1\in\Z_p^k$, so that $Av = \det(A)e_1$.  
$\Z$ is dense in $\Z_p$, so there exists $w = (w_1,\ldots,w_k)^t $ $\in\Z^k$ such that $|v_i-w_i|_p < \det(A)/2$ for each index $1\leq i \leq k$.  
Then, by the strong triangle inequality, each entry of $Aw$ differs from the corresponding entry of $Av$ by something with absolute value less than $\max_{1\leq i,j\leq k} \{ |f_i(g_j)|_p\} \det(A)/2 \leq \det(A)/2$.  
In particular, the first entry of $Aw$ has absolute value greater than $\det(A)/2$, and all other entries have absolute value less than this.  

Thus $|f_1(w_1g_1+\cdots + w_kg_k)|_p > $ $\det(A)/2 > $ $|f_i(w_1g_1+\cdots + w_kg_k)|_p$ for all $2\leq i \leq k$.  
$\{ f_2,\ldots, f_k\}$ is a basis for the $\Z_p$-module spanned by $\{ f'|_X : f'\in F\}$, so, given any $f'\in F$, there exist $\beta_2,\ldots,\beta_k\in\Z_p$ such that $f'(w_1g_1+\cdots + w_kg_k) = $ $\sum_{j=2}^k \beta_j f_j(w_1g_1+\cdots + w_kg_k)$.  
But by the strong triangle inequality, the absolute value of this sum is at most $\max_{2\leq j \leq k} \{ |\beta_j|_p|f_j(w_1g_1+\cdots + w_kg_k)|_p\} \leq $ $\det(A)/2 < $ $|f(w_1g_1+\cdots + w_kg_k)|_p$.  
Thus statement (1) is true for the element $w_1g_1+\cdots w_kg_k\in G$.  
\end{proof}

The next result describes the subsets of $G^{*p}$, the relative-topology closures of the the spans of which are all of $G^{*p}$.  

\begin{prop}\label{PROP:spanning-states}
Let $p$ be a prime, let $G$ be a torsion-free abelian group with $p$-adic pseudometric $d_p$, and let $F$ be a subset of $G^{*p}$.  
The following are equivalent:  
\begin{enumerate}
\item  For each $g\in G$, there exists $f\in F$ such that $|f(g)|_p = d_p(g,0)$.  
\item  $G^{*p}$ is the relative-topology closure of the $\Z_p$-span of $F$.  
\end{enumerate}
\end{prop}
\begin{proof}
To prove that (2) implies (1), suppose that $G^{*p}$ is the relative-topology closure of the $Z_p$-span of $F$, and pick $g\in G$.  
By Corollary \ref{COR:state-hits-distance}, there exists a $p$-adic functional $f$ on $G$ such that $|f(g)|_p = d_p(g,0)$.  
We can write $f(g)$ in terms of $\{ f'(g) : f'\in F\}$:  
\begin{align*}
f(g) & = \alpha_1f_1(g) + \cdots + \alpha_kf_k(g),
\end{align*}
where $\alpha_i\in\Z_p$ and $(f_i)_{i=1}^k$ are elements of $F$.  
For each $i\leq k$, $|f_i(g)|_p \leq d_p(g,0)$.  
If the stronger statement $|f_i(g)|_p < d_p(g,0)$ were true for all $i\leq k$, then, by the strong triangle inequality, it would follow that 
\begin{align*}
|f(g)|_p & \leq \max_{i\leq k} \{ |\alpha_i|_p|f_i(g)|_p\} < d_p(g,0),
\end{align*}
which is a contradiction.  
Thus $|f_i(g)|_p = d_p(g,0)$ for at least one $f_i\in F$.  

To prove that (1) implies (2), suppose that statement (1) is true, but the relative-topology closure of the $\Z_p$-span of $F$ is properly contained in $G^{*p}$.  
Then pick a $p$-adic functional $f$ on $G$ that is not in the relative-topology closure of the $\Z_p$-span of $F$.  
By Lemma \ref{LEM:independence}, there exists $g\in G$ such that $|f(g)|_p > $ $|f'(g)|_p$ for all $f'\in F$.  
But this is impossible because, by statement (1), there exists $f'\in F$ such that $|f'(g)|_p = d_p(g,0)$, and $|f(g)|_p$ cannot exceed this value because $f$ is a $p$-adic functional on $G$.  
\end{proof}

\subsection{Double duals}\label{SUBSEC:double-duals}

If $p$ is a prime and $G$ is a torsion-free abelian group, then the dual space $G^{*p}$ of $p$-adic functionals is also a torsion-free abelian group, and so has its own dual space $G^{*p*p}$.  
There is a natural homomorphism $\Phi^p$ from $G$ to $G^{*p*p}$ (Definition \ref{DEF:phi-p}, below), the image of which is dense in $G^{*p*p}$ (Theorem \ref{THM:density}, below).  

\begin{defn}\label{DEF:g-hat}
Let $p$ be a prime, let $G$ be a torsion-free abelian group, and let $g\in G$.  
Define a function $\hat{g}^p : G^{*p}\to \Z_p$ by $\hat{g}^p(f) = f(g)$ for $f\in G^{*p}$.  
\end{defn}
$\hat{g}^p$ is clearly a group homomorphism into $\Z_p$, so is an element of $G^{*p*p}$.  

\begin{defn}\label{DEF:phi-p}
Let $p$ be a prime and let $G$ be a torsion-free abelian group.  
Define a function $\Phi^p$ by 
\begin{align*}
\Phi^p : G & \to G^{*p*p} \\
g & \mapsto \hat{g}^p.
\end{align*}
\end{defn}
Note that $\Phi^p$ is a group homomorphism: $\widehat{g+h}^p(f) = f(g+h) = f(g) + f(h) = \hat{g}^p(f) + \hat{h}^p(f)$.  
Its kernel certainly contains $G_p$, and, by Corollary \ref{COR:state-hits-distance}, $G_p$ is the entire kernel.  
Let $\hat{G}^p$ denote the image of $\Phi^p$: $\hat{G}^p := \{ \hat{g}^p : g\in G\}\subset G^{*p*p}$.  

\begin{remark}\label{REM:not-tensor-product}
By tensoring with $\Z_p$ over $\Z$, we can embed any torsion-free abelian group $G$ into a free $\Z_p$-module.  
Such embeddings have been studied in \cite{M:tensor-product} and \cite{K:matrices}; they are similar to the homomorphism $\Phi^p : G \to G^{*p*p}$, but are not generally the same.  
If $G_p$ is not trivial, then $\Phi^p$ is not an embedding, and even when it is an embedding, it is not necessarily equivalent to the tensor product embedding.  
Example \ref{EX:favourite} illustrates this fact: in that example, $G$ is a rank-$2$ group, so its tensor product with $\Z_3$ has $\Z_3$-rank equal to $2$, while its $p$-adic dual group $G^{*p}$ is generated over $\Z_3$ by a single row, and so $G^{*p*p}$ has rank equal to $1$.  
%
\end{remark}

The main result of this section is the following.  

\begin{thm}\label{THM:density}
Let $p$ be a prime and let $G$ be a torsion-free abelian group with $p$-adic pseudometric $d_p$.  
Then $\Phi^p : G\to\hat{G}^p\subseteq G^{*p*p}$ is a pseudo-isometry, and $d_p(g,0) = d_p^{*p*p}(\hat{g}^p,0) = \max_{f\in G^{*p}}\{ |\hat{g}^p(f)|_p\}$, where $d_p^{*p*p}$ denotes the $p$-adic metric on $G^{*p*p}$.  
Furthermore $\hat{G}^p$ is dense in $G^{*p*p}$ in the relative topology.  
\end{thm}
\begin{proof}
Pick $g\in G$, and let us show that $d_p^{*p*p}(\hat{g}^p,0) = d_p(g,0) = \max_{f\in G^{*p}} \{ |\hat{g}^p(f)|_p\}$.  
Certainly for any $f\in G^{*p}$, it is true that $|f(g)|_p\leq d_p(g,0)$, so it is enough to show that this value is attained for some $p$-adic functional $f$.  

By Corollary \ref{COR:state-max-attained} applied to $\hat{g}^p$, there exists $f\in G^{*p}$ with $d_p^{*p}(f,0) = 1$ such that 
\[
d_p^{*p*p}(\hat{g}^p,0) = |\hat{g}^p(f)|_p = |f(g)|_p \leq d_p(g,0) = d_p(g,0)d_p^{*p}(f,0) \leq |f(g)|_p,
\]
where the last inequality follows from Lemma \ref{LEM:state-space-metric}.  
Thus $d_p^{*p*p}(\hat{g}^p,0) = d_p(g,0) = |\hat{g}^p(f)|_p = |f(g)|_p$.  

Now let us show that $\hat{G}^p$ is dense in $G^{*p*p}$ in the relative topology.  
Pick $f\in G^{*p}$.  
By Corollary \ref{COR:state-max-attained}, there exists $g\in G$ with $d_p^{*p}(f,0) = |f(g)|_p = |\hat{g}^p(f)|_p$; by Proposition \ref{PROP:spanning-states}, this implies that the $\Z_p$-span of $\hat{G}^p$ is dense in $G^{*p*p}$ in the relative topology.  
So given $\phi\in G^{*p*p}$, a finite subset $F = \{ f_1,\ldots,f_k\}\subseteq G^{*p}$, and $\epsilon > 0$, there exist $g_1,\ldots, g_l\in G$ and $\alpha_1,\ldots, \alpha_l\in \Z_p$ such that 
\begin{align*}
|(\sum\alpha_j\hat{g}^p_j)(f_i) - \phi(f_i)|_p < \epsilon
\end{align*}
for each $1\leq i\leq k$.  

For $1\leq j\leq l$ choose $a_j\in \Z$ such that $|a_j-\alpha_j|_p < \epsilon$.  
Then $\sum a_j\hat{g}^p_j\in \hat{G}^p$ because $\hat{G}^p$ is a group, and 
\begin{align*}
&   |(\sum a_j\hat{g}^p_j )(f_i) - \phi(f_i)|_p \\
 \leq & \max\{ |(\sum \alpha_j\hat{g}^p_j)(f_i)-\phi (f_i)|_p, |(\sum a_j\hat{g}^p_j)(f_i) - (\sum \alpha_j\hat{g}^p_j )(f_i)|_p\} \\
 \leq & \max \{ \epsilon, |(a_1-\alpha_1)\hat{g}^p_1(f_i)|_p,\ldots, |(a_l-\alpha_l)\hat{g}^p_l(f_i)|_p\} \\
\leq & \max \{ \epsilon, \epsilon, \ldots, \epsilon\} = \epsilon.  
\end{align*}
Thus $\hat{G}^p$ is dense in $G^{*p*p}$ in the relative topology.  
\end{proof}

\section{Pro-$p$ completions}\label{SEC:completions}

Readers who are familiar with profinite groups might have noticed that the double dual $G^{*p*p}$ of a torsion-free abelian group $G$ shares important features with the pro-$p$ completion of $G$.  
This is not a coincidence---the two objects are, in fact, the same.  

\begin{defn}\label{DEF:pro-p}
A topological group $G$ is called a \defemph{pro-$p$ group} if it is isomorphic as a topological group to the inverse limit of an inverse system of finite $p$-groups.  
\end{defn}

The topology in Definition \ref{DEF:pro-p} may require some explanation.  
Any inverse limit of an inverse system of groups has a standard realization as a subgroup of the direct product:  
\[
\varprojlim (G_\lambda)_{\lambda\in\Lambda} = \{ (x_\lambda)_{\lambda\in\Lambda} \in \prod_{\lambda\in\Lambda} G_\lambda : \phi_{\mu\lambda}(x_\lambda) = x_\mu \text{ for all } \lambda\leq \mu\},  
\]
where, for $\lambda\leq \mu$, $\phi_{\mu\lambda} : G_\lambda\to G_\mu$ is the corresponding map in the inverse system.  
In the pro-$p$ case, the groups $G_\lambda$ in the direct product are finite $p$-groups, and hence are topological groups with the discrete topology.  
The topology on the inverse limit is the relative topology that it inherits from the direct product.  

Given an arbitrary group $G$, consider the set of all normal subgroups $N$ of $G$ for which $G/N$ is a finite $p$-group.  
This is a partially ordered set under inclusion, which means that there is an inverse system of natural homomorphisms $\phi_{\mu\lambda} : G/N_\lambda \to G/N_\mu$ between the quotients, where $N_\mu\leq N_\lambda$.  
The \defemph{pro-$p$ completion} $\mathcal{K}_p(G)$ of $G$ is the inverse limit of this inverse system.  
There is a natural embedding $\Psi^p : G\to \mathcal{K}_p(G)$, which sends $g\in G$ to $(gN_\lambda)_{\lambda\in\Lambda}$, and the image of $G$ under this embedding is a dense subgroup.  

$\mathcal{K}_p(G)$ is clearly a pro-$p$ group.  
Moreover it satisfies the following universal property with respect to $G$:  if there exists a pro-$p$ group $H$ and a group homomorphism $\theta : G\to H$, then there exists a continuous homomorphism $\eta : \mathcal{K}_p(G)\to H$ such that $\theta = \eta\circ\Psi^p$.  

If $G$ is a torsion-free abelian group, then $G^{*p*p}$ is a free $\Z_p$-module, and hence is a pro-$p$ group (and the pro-$p$ topology agrees with the relative topology on $G^{*p*p}$ \cite[Proposition 2.2.1, parts (a) and (b)]{RZ:profinite}).  
Thus the universal property guarantees the existence of a continuous homomorphism $\eta : \mathcal{K}_p(G)\to G^{*p*p}$ such that $\eta(\Psi^p(g)) = \hat{g}^p$ for all $g\in G$.  
Then in order to show that $G^{*p*p}$ is the pro-$p$ completion, it will be necessary to find a continuous homomorphism in the other direction.  
Let us do so by finding a compatible system of homomorphisms from $G^{*p*p}$ into all of the finite $p$-group quotients of $G$, and then use the universal property of the inverse limit to construct the required homomorphism.  

\begin{lemma}\label{LEM:quotient}
Let $p$ be a prime, let $G$ be a torsion-free abelian group, and let $H$ be a subgroup of $G$ of index $p^k$ for some $k\in\N$.  
Then there exists a unique continuous homomorphism $\theta_H : G^{*p*p}\to G/H$ such that $\theta_H(\hat{g}^p) = g+H$ for all $g\in G$.  
\end{lemma}
\begin{proof}
There exists $m\in\N\cup\{ 0\}$ such that the maximum order of any element $g+H\in G/H$ is $p^m$.  
By Lemma \ref{LEM:subgroup-separation}, if $g,g'\in G$ with $d_p(g,g')\leq p^{-m}$, then $g$ and $g'$ are in the same $H$-coset.  
Let $\{g_1,\ldots, g_{p^k}\}$ be a complete set of coset representatives of $G/H$, with $g_1\in H$, and let $F$ be a finite set of $p$-adic functionals such that, for each $2\leq i\leq p^k$, there exists $f\in F$ such that $|f(g_i)|_p > p^{-m} \geq |f(h)|_p$ for all $h\in H$.  
(Such a set $F$ exists by Proposition \ref{PROP:separate-from-subgroup}.)  

For $\phi\in G^{*p*p}$, consider the open neighbourhood $U_\phi$ of $\phi$ defined by $U_\phi = \prod_{f\in G^{*p}} U_f$, where $U_f = \{ \beta\in \Z_p : |\beta-\phi(f)|_p \leq p^{-m}\}$ if $f \in F$, and for $f\notin F$, $U_f = \Z_p$.  
Then $U_\phi$ is an open subset of $\Z_p^{G^{*p}}$, so its intersection $U_\phi'$ with $G^{*p*p}$ is open in the relative topology on $G^{*p*p}$.  
$\hat{G}^p$ is dense in $G^{*p*p}$ in the relative topology, so there exists some $g\in G$ such that $\hat{g}^p\in U_{\phi}'$.  
Set $\theta_H(\phi) = g+H$, and let us show that this is well-defined.  

Let us show that, if we had chosen a different element $g'\in G$ with $\hat{g}'^p\in U_\phi'$, then $g+H = g'+H$.  
To see this, note that $\hat{g}'^p\in U_\phi'$ implies that $|f(g-g')|_p\leq p^{-m}$ for all $f\in F$.  
Suppose that $g_j$ is the chosen representative of the coset $g-g' + H$; this means that $g_j - g + g'\in H$.  
If $j = 1$, then $g_1\in H$, so $g-g'\in H$ as required.  
Otherwise, pick $f\in F$ such that $|f(g_j)|_p > p^{-m} \geq |f(h)|_p$ for all $h\in H$, and note that, by the strong triangle inequality, 
\[
p^{-m} < |f(g_j)|_p \leq \max \{ |f(g-g')|_p,|f(g_j-g+g')|_p\}.  
\]
But $g_j-g+g'\in H$, so $|f(g_j-g+g')|_p \leq p^{-m}$, while $|f(g-g')|_p \leq \max\{ |f(g)-\phi(f)|_p,|f(g')-\phi(f)|_p\}$, which is less than or equal to $p^{-m}$ because $\hat{g}$ and $\hat{g}'$ are both in $U_\phi$.  
This is a contradiction, so $g-g'\in H$ after all, and $\theta_H$ is well defined.  

To see that $\theta_H$ is a group homomorphism, choose $\phi_1,\phi_2\in G^{*p*p}$ and observe that, if $g_1,g_2\in G$ satisfy $\hat{g}_1^p\in U_{\phi_1}$ and $\hat{g}_2^p\in U_{\phi_2}$, then $\widehat{g_1+g_2}^p\in U_{\phi_1+\phi_2}$ by the strong triangle inequality.  

To see that $\theta_H$ is continuous, pick $g+H\in G/H$ and consider an element in its pre-image set: $\phi\in \theta_H^{-1}(\{ g+H\})$.  
Then by the argument above, the open neighbourhood $U_\phi'\subseteq \theta_H^{-1}(g+H)$, so the pre-image of $g+H$ is open and $\theta_H$ is continuous.  

To see that $\theta_H$ is unique with these properties, suppose that $\theta_H': G^{*p*p}\to G/H$ is another continuous group homomorphism such that $\theta_H'(\hat{g}^p) = g+H$ for all $g\in G$.  
Then $\theta_H-\theta_H' : G^{*p*p}\to G/H$ is a continuous group homomorphism.  
To see that it is a group homomorphism is easy.  
To see that it is continuous, choose $g+H\in G/H$ and suppose $\phi\in (\theta_H-\theta_H')^{-1}(g+H)$.  
Say $\theta_H(\phi) = g_1+H$ and $\theta_H'(\phi) = g_2+H$; then $g_1+g_2 +H = g+H$.  
As $\theta_H$ and $\theta_H'$ are continuous, there exists open neighbourhoods $U_1,U_2$ of $\phi$ such that $U_1\subseteq \theta_H^{-1}(g_1+H)$ and $U_2\subseteq \theta_{H}'^{-1}(g_2+H)$; the open neighbourhood $U_1\cap U_2$ of $\phi$ then has the property that $U_1\cap U_2\subseteq (\theta_H-\theta_H')^{-1}(g+H)$.  
Thus $\theta_H-\theta_H'$ is continuous.  

The kernel of $\theta_H-\theta_H'$ is then a closed subset of $G^{*p*p}$, and, as it contains the dense subset $\hat{G}^p$, it must be everything.  
\end{proof}

\begin{thm}\label{THM:pro-p}
Let $p$ be a prime and let $G$ be a torsion-free abelian group.  
Then there is an isomorphism $\eta : \mathcal{K}_p(G) \to G^{*p*p}$ with the properties that $\eta(\Psi^p(g)) = \hat{g}^p$ for all $g\in G$ and $\eta$ is a homeomorphism.  
\end{thm}
\begin{proof}
The existence of a continuous group homomorphism $\eta$ such that $\eta(\Psi^p(g)) = \hat{g}^p$ is guaranteed by the universal property of the pro-$p$ completion.  
To see that $\eta$ is an isomorphism, let us construct an inverse homomorphism.  

Consider the set of subgroups $H$ of $G$ such that $[G:H] = p^k$ for some $k\in\N\cup\{ 0\}$.  
By Lemma \ref{LEM:quotient}, for each of these subgroups, there exists a continuous homomorphism $\theta_H:G^{*p*p}\to G/H$; moreover these homomorphisms are easily seen to be compatible with the natural homomorphisms between the finite $p$-group quotients of $G$.  
Therefore by the universal property of the inverse limit, there is a continuous group homomorphism $\rho : G^{*p*p}\to \mathcal{K}_p(G)$.  
Moreover, for any $g\in G$, the element $\hat{g}^p$ has image $g+H$ under $\theta_H$, which implies that $\rho(\hat{g}^p) = \Psi^p(g)$.  

Let us show that $\rho$ is an inverse for $\eta$.  
Note that $\rho\circ \eta$ is continuous, and, for $g\in G$, it is true that $\rho(\eta(\Psi^p(g))) = \rho(\hat{g}^p) = \Psi^p(g)$.  
Thus $\rho\circ\eta$ is the identity on the dense subgroup $\{ \Psi^p(g) : g\in G\}$, so $\rho\circ\eta - id_{\mathcal{K}_p(G)}$ is a continuous group homomorphism that contains a dense subgroup in its kernel, and hence is trivial.  

A similar argument shows that $\eta\circ\rho = id_{G^{*p*p}}$.  
\end{proof}

%
%

\section{Classification}\label{SEC:classification}

The theory of $p$-adic functionals can be used to classify finite-rank torsion-free abelian groups.  
Specifically, the collection of homomorphisms $\Phi^p : G\to G^{*p*p}$ for a torsion-free abelian group $G$, as $p$ ranges over the set of all primes, gives enough information to recover the isomorphism class of $G$.  

Let $P$ denote the set consisting of all integer primes and let $G$ be a torsion-free abelian group.  
Pick a maximal independent subset $\mathcal{B}$ of $G$.  

For an element $(\alpha_\lambda)_{\lambda\in\Lambda} \in \Z_p^{\Lambda}$, define the sup-norm $\| (\alpha_\lambda)_{\lambda\in\Lambda} \|_p = \sup_{\lambda\in\Lambda} \{ |\alpha_\lambda|_p\}$.  
It is easy to verify that $\|\cdot \|_p$ is indeed a norm, and that it satisfies the strong triangle inequality.  

$G^{*p*p}$ is isomorphic to $\Z_p^{\Lambda_p}$ for some index set $\Lambda_p$; if we identify $G^{*p*p}$ with this latter group, then the homomorphism $\Phi^p$ sends elements of $G$ to $\Lambda_p$-tuples in $\Z_p^{\Lambda_p}$.  
Moreover, by Theorem \ref{THM:density}, $\Phi^p$ is a pseudo-isometry, meaning that $d_p(g,0) = \| \Phi^p(g)\|_p$.  
This is a powerful fact: it implies that $G$ is entirely determined by the values $(\Phi^p(g))_{p\in P}$ as $g$ ranges over the elements of $\mathcal{B}$.  
Indeed, $G$ is isomorphic to a subgroup of $\langle\mathcal{B}\rangle\otimes \Q$, so in order to specify $G$ it is enough to be able to say, for each element $g = a_1g_1+\cdots +a_kg_k\in \langle\mathcal{B}\rangle\subseteq G$ and $m\in\N$, whether or not $g$ is in $mG$.  
But $g\in mG$ if and only if $|m|_p \geq d_p(g,0) = \| \Phi^p(g)\|_p = \| a_1\Phi^p(g_1) + \cdots + a_k\Phi^p(g_k)\|_p$ for all $p\in P$.  
This is summarized in the following proposition.  

\begin{prop}\label{PROP:classification-1}
Let $G$ be a torsion-free abelian group, let $P$ denote the set of integer primes, and, for each $p\in P$, identify $G^{*p*p}$ with the group $\Z_p^{\Lambda_p}$ for some index set $\Lambda_p$.  
If $\mathcal{B}$ is any maximal independent subset of $G$, then the values $(\Phi^p(g))_{p\in P, g\in\mathcal{B}}$ completely determine $G$.  
\end{prop}

In the case that $G$ has finite rank, the data that specify $G$ can be stated more explicitly in terms of the following definition.  
\begin{defn}\label{DEF:group-determined}
Let $n\in \N$, let $P$ be the set of all primes and let $(A_p)_{p\in P}$ be a sequence of matrices with $A_p\in M_{n_p\times n}(\Z_p)$.  
Then the \defemph{group determined by $(A_p)_{p\in P}$} is the torsion-free abelian group 
\[
\mathcal{G}((A_p)_{p\in P}) := \{ v\in \Q^n : A_pv\in \Z_p^{n_p} \text{ for all } p\in P\}.
\]
\end{defn}

\begin{prop}\label{PROP:classification-2}
Let $P$ denote the set of integer primes and let $G$ be a torsion-free abelian group with finite rank, $n$.  
Choose
\begin{enumerate}
\item  a maximal independent set $\{ g_1,\ldots,g_n\}\subset G$; and
\item  for each $p\in P$, a $\Z_p$-basis $\{ f_1^{(p)},\ldots, f_{n_p}^{(p)}\}$ for $G^{*p}$.  
\end{enumerate}
For each $p\in P$, let $A_p = (f_i^{(p)}(g_j)) \in M_{n_p\times n}(\Z_p)$.  
Then $G$ is isomorphic to $\mathcal{G}((A_p)_{p\in P})$.  
\end{prop}

Thus the finite-rank torsion-free abelian group $G$ is completely described by a sequence, indexed by the prime numbers, of matrices $A_p$ with entries in $\Z_p$.  
The rows of the $p$th matrix correspond to elements of a basis for $G^{*p}$.  
This fact places a restriction on the structure of $A_p$; this restriction is expressed in terms of the following definition.  

\begin{defn}\label{DEF:column-span}
Let $p$ be a prime, let $m,n\in\N$, and let $A\in M_{m\times n}(\Z_p)$.  
Then define the \defemph{column span} of $A$ to be 
\[
C(A) := \spa_\Z \{ v/\| v\|_p : v \text{ is a non-zero column of } A\},
\]
where, for a vector $v\in\Z_p^{n}$, $\| v\|_p$ is the maximum of the $p$-adic absolute value of its entries.  
\end{defn}
Then, by Theorem \ref{THM:density}, the matrices $A_p$ in Proposition \ref{PROP:classification-2} have the property that the column span $C(A_p)$ of $A_p$ is dense in $\Z_p^{n_p}$.  
This implies in particular that $n_p\leq n$ for all primes $p$.  

\begin{example}\label{EX:favourite-matrix-sequence}
Let $G$ be the group from Example \ref{EX:favourite}.  
Then for $p\neq 3$, by choosing the most obvious basis for $G^{*p}$, we can arrange for the matrix $A_p$ to be $\twotwo{1}{0}{0}{1}$.  
The $3$-adic functional described in that example gives rise to the matrix $A_3 = (1-\lambda,1)$, where $\lambda$ is the root of $x^2-5x+3=0$ for which $|\lambda|_3<1$.  
\end{example}

\begin{example}\label{EX:two-three-five}
Let $G$ be the subgroup of $\Q^2$ generated by $\Z[1/2]\oplus \Z[1/3]$ and $\Z[1/5]{{1}\choose {1}}$.  
Then the obvious choices of bases for $G^{*2}$, $G^{*3}$, and $G^{*5}$ yield matrices $A_2 = (0,1)$, $A_3 = (1,0)$, and $A_5 = (1,-1)$.  
For $p > 5$, by choosing the most obvious basis for $G^{*p}$, we can arrange for the matrix $A_p$ to be $\twotwo{1}{0}{0}{1}$.  
\end{example}

\begin{defn}\label{DEF:factored}
Let $P$ denote the set of primes and let $n\in\N$.  
Suppose that $(m_p)_{p\in P}$ is a sequence of non-negative integers and $(A_p)_{p\in P}$ is a sequence of matrices satisfying  
\begin{enumerate}
\item  $A_p\in M_{m_p\times n}(\Z_p)$;
\item  either $m_p = 1$ and $A_p$ is the zero row in $M_{1\times n}(\Z_p)$, or else $C(A_p)$ is dense in $\Z_p^{m_p}$ (see Definition \ref{DEF:column-span}); and
\end{enumerate}
Then let us refer to the matrices $(A_p)_{p\in P}$ as a \defemph{factored form}.  
If, for some rank-$n$ torsion-free abelian group $G$, it is true that $G\cong \mathcal{G}((A_p)_{p\in P})$, then let us refer to the matrices $(A_p)_{p\in P}$ as a \defemph{factored form of $G$}.  
\end{defn}

\begin{remark}
The matrices in the factored form of $G$ are related to Malcev's matrices \cite{M:matrices} (Malcev's matrices were not originally expressed as matrices with entries in $\Z_p$, but in \cite{K:summary} Kostromina demonstrates that they can be equivalently described that way).  
Malcev's matrices $(M_p)_{p\in P}$ are square matrices, and they satisfy the condition that the first $n-n_p$ columns of $A_pM_p^t$ are zero, where $A_p$ comes from a factored form derived from an appropriate maximal independent subset of $G$.  

The factored form is also related to the matrices of Kurosch \cite{K:matrices}.  
Following the description in \cite[Chapter XIII, Section 93]{F:book}, the matrices of Kurosch come from the tensor product $G\otimes \Z_p$.  
This decomposes (not necessarily uniquely) as a direct sum of a reduced part and its non-reduced (i.e. divisible) subgroup.  
The divisible part is a $\Q_p$-vector space and the reduced part is a free $\Z_p$-module.  
Let $\mathcal{B}$ be a maximal independent subset of $G$.  
The coefficients in the $i$th row of the $p$th Kurosch matrix express the $i$th element of $\mathcal{B}$ as a combination of basis elements for the divisible and reduced parts of $G\otimes \Z_p$.  

One main result of the present work is that the information from the divisible subgroup of $G\otimes \Z_p$ is extraneous; that is, the last $n_p$ columns of the Kurosch matrices are enough to determine $G$ completely.  
\end{remark}

It is possible to characterize the homomorphisms between factored forms.  
\begin{prop}\label{PROP:factored-homomorphisms}
Let $(A_p)_{p\in P}$ and $(B_p)_{p\in P}$ be two factored forms, with each matrix $A_p\in M_{m_p\times m}(\Z_p)$ and each matrix $B_p\in M_{n_p\times n}(\Z_p)$.  
Then a matrix $V\in M_{n\times m}(\Q)$ represents a group homomorphism from $\mathcal{G}((A_p)_{p\in P})$ to $\mathcal{G}((B_p)_{p\in P})$ if and only if, for all $p\in P$, every row of $B_pV$ lies in the $\Z_p$-module spanned by the rows of $A_p$.  
\end{prop}
\begin{proof}
The ``only if'' part of the statement is clear, because left multiplication by any row $R$ of $B_p$ is a $p$-adic functional on $\mathcal{G}((B_p)_{p\in P})$, so the composition $g \mapsto Vg \mapsto RVg$ is a $p$-adic functional on $\mathcal{G}((A_p)_{p\in P})$, and hence must correspond to a $\Z_p$ combination of the rows of $A_p$.  

For the ``if''' part of the statement, suppose that, for all primes $p$, every row of $B_pV$ lies in the $\Z_p$-module spanned by the rows of $A_p$.  
Pick $g\in \mathcal{G}((A_p)_{p\in P})$.  
To show that left multiplication by $V$ is a homomorphism, it is enough to show that, if $g\in p^k\mathcal{G}((A_p)_{p\in P})$ for some prime $p$ and $k\geq 0$, then $Vg \in p^k\mathcal{G}((B_p)_{p\in P})$.  

So suppose $g\in p^k\mathcal{G}((A_p)_{p\in P})$.  
By Definition \ref{DEF:group-determined}, this means that $A_pg\in p^k\Z_p^{m_p}$.  
Thus if $R'$ is a row of $A_p$, then $R'g \in p^k\Z_p$.  
As any row $R$ of $B_pV$ is in the $\Z_p$-span of the rows of $A_p$, this means that $Rg\in p^k\Z_p$ as well, so $B_pVg\in p^k\Z_p^{n_p}$, and by Definition \ref{DEF:group-determined} $Vg\in  p^k\mathcal{G}((B_p)_{p\in P})$.  
\end{proof}

We are now in a position to state Theorem \ref{THM:classification-3}, the main classification result.  
It says that every finite-rank torsion-free abelian group has a factored form, and gives a characterization of isomorphisms between factored forms.  
The first statement is a consequence of Proposition \ref{PROP:classification-2} and the discussion following it, while the second statement is a consequence of Proposition \ref{PROP:factored-homomorphisms}.  

\begin{thm}\label{THM:classification-3}
Any finite-rank torsion-free abelian group $G$ has a factored form.  

Moreover, if $(m_p,A_p)_{p\in P}$ and $(n_p,B_p)_{p\in P}$ are two factored forms, then $\mathcal{G}((A_p)_{p\in P})$ is isomorphic to $\mathcal{G}((B_p)_{p\in P})$ if and only if the following conditions hold.  
\begin{enumerate}
\item  $m_p = n_p$ for all $p\in P$;
\item  there exists $V\in GL(n,\Q)$ such that, for all $p\in P$, $B_p = U_pAV$ for some $U_p\in GL(n_p,\Z_p)$.  
\end{enumerate}
\end{thm}

The matrix $V$ in Theorem \ref{THM:classification-3} should be thought of as a change of basis matrix that represents a new maximal independent subset in terms of the old one.  
The matrices $U_p$ should be thought of as change of basis matrices that represent a new basis of the $p$-adic dual group in terms of the old basis.  

\begin{remark}\label{REM:isomorphism-check}
In order to check that two factored forms represent isomorphic groups, it is necessary to check infinitely many matrix equations.  
Specifically, condition (2) in the definition of a factored form (Definition \ref{DEF:factored}) implies that, if $B_p$ is not the zero row, then $B_pB_p^t$ is invertible over $\Q_p$.  
Thus the equation in Theorem \ref{THM:classification-3} can be rearranged to obtain 
\begin{align}\label{EQ:system}
A_pVB_p^t(B_pB_p^t)^{-1} & = U_p^{-1}.
\end{align}

So two factored forms represent isomorphic groups if and only if there exists a matrix $V\in GL(n,\Q)$ such that, for each prime $p$, the matrix on the left hand side of Equation \ref{EQ:system} is an element of $GL(n_p,\Z_p)$.  
\end{remark}

\bibliographystyle{abbrv}

\end{document}